\documentclass[]{interact}
\usepackage{geometry}
\usepackage{epstopdf}% To incorporate .eps illustrations using PDFLaTeX, etc.
\usepackage[caption=false]{subfig}% Support for small, `sub' figures and tables
\usepackage[T1]{fontenc}
\usepackage{bm}
\usepackage{MnSymbol}
\usepackage{amssymb}
\usepackage{soul}
\usepackage{setspace}
\usepackage{tikz-cd}
\usepackage{tikz}
\usepackage{hyperref}
\usepackage{amsmath}
\usepackage[backend=biber,backref=true]{biblatex}
\usepackage{comment}
\makeatletter% @ becomes a letter
\def\NAT@def@citea{\def\@citea{\NAT@separator}}% Suppress spaces between citations using natbib.sty
\makeatother% @ becomes a symbol again

\theoremstyle{plain}% Theorem-like structures provided by amsthm.sty
\newtheorem{theorem}{Theorem}[section]
\newtheorem{lemma}[theorem]{Lemma}

\newtheorem{proposition}[theorem]{Proposition}

\theoremstyle{definition}
\newtheorem{definition}[theorem]{Definition}

\theoremstyle{remark}
\newtheorem{remark}{Remark}

\begin{document}

%\articletype{Toric ALE Gravitational instantons}% Specify the article type or omit as appropriate

\title{\vspace{-0.5 cm}Gauge theory and symplectic structures}
\author{
\name{Partha Ghosh}
%\affil{\textsuperscript{a}Taylor \& Francis, 4 Park Square, Milton Park, Abingdon, UK; \textsuperscript{b}Institut f\"{u}r Informatik, Albert-Ludwigs-Universit\"{a}t, Freiburg, Germany}
}
\newcommand{\Addresses}{{% additional braces for segregating \footnotesize
  \bigskip
  \footnotesize

  \textsc{Sorbonne Université, Université Paris Cité, CNRS, IMJ-PRG, F-75005 Paris, France}\par\nopagebreak
  \textit{E-mail address}: \texttt{Partha.Ghosh@imj-prg.fr}
}}
\maketitle
\begin{abstract} Motivated by the problem of finding obstructions to the existence of symplectic structures in dimensions higher than four, we introduce an elliptic system of equations on almost Hermitian manifolds that reduces to the Seiberg--Witten equations with Taubes' perturbation in dimension four. To define the equations, one needs to choose a spin$^{\mathbb C}$ structure on the manifold. We prove that the system has index zero for the canonical spin$^{\mathbb C}$ structure. Moreover, on almost Kähler (symplectic) manifolds, we construct a canonical solution and prove its transversality with large values of the perturbation parameter, while in dimension six we obtain uniqueness of the canonical solution under a certain assumption. We also show that on Kähler manifolds the equations reduce, under a natural ansatz, to the vortex equations. These results provide evidence toward a possible higher dimensional gauge theory capable of producing invariants of almost complex structures and ultimately, obstructions to compatible symplectic forms.

\end{abstract}
\tableofcontents
\section{Introduction}
\subsection{Background}
Let $M$ be a smooth orientable manifold of even dimension $n=2m$. A fundamental existence problem in symplectic topology asks the following:
\begin{center}
    \textit{under which conditions a given cohomology class $a\in H^2(M,\mathbb{R})$ can be represented by a symplectic form ?}
\end{center}
The obvious necessary conditions are the existence of an almost complex structure on $M$ and in the case of a closed manifold, the condition $a^m\neq 0$. In the open case a theorem of Gromov \cite{Gromov1,Gromov2} asserts that the existence of an almost complex structure is also sufficient. Therefore, the natural improvement of the existence question is the following:
\begin{center}
    \textit{given a cohomology class $a\in H^2(M,\mathbb{R})$ with $a^m\neq 0$ on an almost complex closed manifold $(M,J)$, can it be represented by a symplectic form ?}
\end{center}
The only known counter examples exist in dimension four, based on Seiberg--Witten theory (see Taubes \cite{Taubes1}). One explicit example is 
\begin{align*}
    M=\mathbb{CP}^2\# \mathbb{CP}^2\#\mathbb{CP}^2 
\end{align*}
$M$ has a real cohomology class $a$ satisfying $a^2\neq 0$ and an almost complex structure but $M$ does not admit any symplectic form. In higher dimensions the existence problem is completely open.\par
In this article, we introduce an elliptic system of equations on an almost complex manifold $(M,J)$ of arbitrary even dimension $n=2m$. When $n=4$, these equations reduce to the classical Seiberg--Witten equations. The equations arise from the following fantasy: given a $\mathrm{spin}^{\mathbb{C}}$ structure on $M$, we hope to ``count'' solutions of these equations and thereby obtain an invariant of the isotopy class of the almost complex structure $J$. If the manifold admits a symplectic structure compatible with $J$, we expect that this invariant equals $1$ for the canonical $\mathrm{spin}^{\mathbb{C}}$ structure induced by $J$. Under these assumptions, if one could find a manifold $M$ for which the invariant associated with the canonical $\mathrm{spin}^{\mathbb{C}}$ structure is not equal to $1$, then $M$ would not admit a symplectic form compatible with $J$. At present, however, this fantasy remains out of reach, especially since the analysis appears to be considerably more complicated (see \S\S\ref{estimates}) than in the four-dimensional Seiberg--Witten theory. Nevertheless, we present evidence, especially in dimensions $6$, suggesting that this program may indeed be feasible.\par 
\subsection{The equations and overview of the main results}
We begin by fixing notation. Let $(M,J)$ be a smooth manifold of dimension $n=2m$ with an almost complex structure $J$. We put a $J$-compatible Riemannian metric $g$ on $M$ to get an almost Hermitian manifold $(M,g,J)$. The fundamental two form $\omega$ is defined to be $\omega(u,v):=g(Ju,v)$.\par
Like in the case of complex manifolds, one defines operators
$\partial$, $\bar\partial$ by setting
$
\partial \alpha = (d\alpha)^{p+1,q}$ \text{and} $\bar\partial \alpha = (d\alpha)^{p,q+1}
$ for a form $\alpha \in \Omega^{p,q}$.
In general,
$
d \neq \partial + \bar\partial,
$
as $d\alpha$ can also have components of type
$(p-1,q+2)$ and $(p+2,q-1)$, denoted
$\bar N\alpha$ and $N\alpha$ respectively.
The operator $N$ is of order zero and can be identified with the
Nijenhuis tensor of $J$; it vanishes if and only if $J$ is integrable.
Observe that $d^2=0$ does not imply $\bar\partial^2=0$, unless
$J$ is integrable. For example, on functions we have $
\bar\partial^2 f = -\bar N(\partial f)$.\par
The almost complex structure $J$ defines a canonical spin$^\mathbb{C}$ structure which gives us the positive and negative spin bundles as
\begin{align*}
    S_+^{\mathrm{can}}= \bigoplus_p \Lambda^{0,2p},\hspace{4 ex}S_-^{\mathrm{can}}= \bigoplus_p \Lambda^{0,2p+1}
\end{align*} 
Strictly speaking, one should write $\Lambda^{p,q}_J$ for $\Lambda^{p,q}$, but this will be done only when the almost complex structure $J$ is not clear from context. The spin bundles associated to the other spin$^{\mathbb{C}}$ structures are obtained by twisting the canonical spin bundle $S^{\mathrm{can}}$ by a complex line bundle $\mathcal{L}: S=S^{\mathrm{can}}\otimes \mathcal{L}$.\par 
We write $\mathrm{cl}: \Lambda^* \xrightarrow{} \text{End}(S)$ for the Clifford action of differential forms on spinors. We follow the conventions of \cite{Friedrich}. In particular, (real) 1-forms act as skew-Hermitian endomorphisms. In dimension $n=2m$, the volume form satisfies $i^m\mathrm{cl}(dv) = \pm 1$ on $S_{\pm}$.\par
Let $L(S)$ denote the line bundle associated to the spin$^\mathbb{C}$ structure and $\mathcal{A}$ denote the set of unitary connections in $L(S)$. Given $A \in \mathcal{A}$, we write $D_A$ for the associated Dirac operator. For $S=S^{\mathrm{can}}\otimes \mathcal{L}, L(S)=K^{-1}\otimes \mathcal{L}^2$. $K$ is the canonical bundle associated to the almost complex structure $J$. \par
A spinor $\phi \in S_\pm$ defines a trace-free Hermitian endomorphism $E_\phi \colon S_\pm \to S_\pm$ via
\begin{equation}
E_\phi(\psi)= \left\langle \psi,\phi \right\rangle\phi - \frac{1}{2^{m-1}} |\phi|^2 \psi
\label{Ephi}
\end{equation}
Taubes' obstruction result \cite{Taubes1} for symplectic four manifolds from Seiberg--Witten equations can be explained as follows. In dimension $4$, Clifford multiplication gives an isomorphism
\begin{align}
    \mathrm{cl}: i\Lambda^2_+ \rightarrow i\mathfrak{su}(S_+)
\end{align}
between the imaginary self-dual $2$-forms and trace-free Hermitian endomorphisms of positive spinors. Given $\phi\in\Gamma(S_+)$, we write $q(\phi)\in i\Omega^2_+$ for the imaginary self-dual $2$-form corresponding to $E_\phi$ under \eqref{Ephi}. The Seiberg--Witten equations with Taubes' perturbation on an almost Hermitian four manifold $(M,g,J)$ for $A\in\mathcal{A},\phi\in\Gamma(S_+)$ are:
\begin{align}
    D_A\phi=0\label{Dirac 4d}\\
    F_A^+-F_{{A^{\mathrm{Ch}}}}^++\frac{ir}{4}\omega=q(\phi)\label{Curvature 4d}
\end{align}
$F_A^+$ denotes the self-dual part of the curvature of $A$, $F_{{A^{\mathrm{Ch}}}}^+$ denotes the self-dual part of the curvature of the anticanonical line bundle $K^{-1}$ with respect to the \textit{Chern} connection ${A^{\mathrm{Ch}}}$ \cite{Paul} (see also~\S\S\ref{symplectic}). The equations are parametrised by a real number $r\in\mathbb{R}$. The gauge group $\mathcal{G}=$ Map $(M,S^1)$ acts on $(A,\phi)$ by pull-back and this action preserves the space of solutions to the equations.\par
The standard Seiberg--Witten theory tells us that the equations \eqref{Dirac 4d}, \eqref{Curvature 4d} are elliptic modulo gauge and the moduli space is compact. Moreover when the spin$^{\mathbb{C}}$ structure is the canonical one induced by $J$, the index is zero. One way to see this is to look at the linearised equations modulo the zeroth order terms. The Dirac equation gives us the operator 
\begin{align}
    \Omega^0\oplus\Omega^{0,2}\xrightarrow[]{\bar\partial+\bar\partial^*} \Omega^{0,1}\label{eq1}
\end{align}
and since $\Lambda^2_+=\mathrm{Re}(\Lambda^0)\langle \omega\rangle\oplus \mathrm{Re}(\Lambda^{0,2})$, modulo gauge the curvature equation gives the following operator
\begin{align}
    \mathrm{Re}(\Omega^{0,1})\xrightarrow[]{d^*+d^+} \mathrm{Re}(\Omega^{0})\oplus \mathrm{Re}(\Omega^0)\omega \oplus \mathrm{Re}(\Omega^{0,2})\label{eq2}
\end{align}
In fact one can identify $\mathrm{Re}(\Omega^{0})\oplus\mathrm{Re}(\Omega^{0})\omega$ with $\Omega^0$ and realise that the operator \eqref{eq2} is secretly nothing but
\begin{align}
    \Omega^{0,1}\xrightarrow[]{\bar\partial+\bar\partial^*} \Omega^{0}\oplus\Omega^{0,2}\label{eq3}
\end{align}
Thereafter, the total index being the sum of index of \eqref{eq1} and the index of \eqref{eq3} is zero. Since the moduli space is compact, the Seiberg--Witten invariant is now worth a signed count of number of solutions. Taubes \cite{Taubes1} showed that when $(M,g,J)$ is almost K\"ahler, i.e., $\omega$ is closed, for $r$ large enough there exists a unique \textit{canonical} solution which is also nondegenerate and therefore up to a choice of orientation the Seiberg--Witten invariant for the canonical spin$^\mathbb{C}$ structure is $1$ (for $b^2_+\geq 2$).\par 
The compactness of the moduli space follows
ultimately from the fact that the left-hand side of the curvature equation \eqref{Curvature 4d}  is directly related to the Weitzenb\"ock formula
\begin{align*}
    D_A^*D_A=\nabla_A^*\nabla_A+\frac{s}{4}+\frac{1}{2}\mathrm{cl}(F_A)
\end{align*}
where $s$ is the scalar curvature of $(M,g)$. Thereafter, on an almost Hermitian $4$-manifold $(M,J)$, using Dirac operators $D_A$ parametrised by $A\in\mathcal{A}$, one can prescribe the projection of the Weitzenb\"ock remainder on $i\big($Re$(\Omega^0)\omega\oplus $Re$(\Omega^{0,2})\big)$ to get an elliptic system modulo gauge.\par
These are the features we want to generalise in our new gauge theoretic equations. We need some preparation before we define the equations on almost Hermitian manifolds $(M,g,J)$ of all even dimensions $n=2m\geq 4$.\par
We define a function $s:\mathbb{N}\rightarrow \{1,i\}$ by
\begin{align*}
s_k = \begin{cases}
	1 & \text{if } k \equiv 0 \text{ or } 3\, \mod 4\\
	i & \text{if } k \equiv 1 \text{ or } 2\, \mod 4	
	\end{cases}
\end{align*}
The point of this definition is that if $\beta_k \in \Omega^k$  then $s_k \mathrm{cl}(\beta_k)$ is a self-adjoint endomorphism of the spin bundle. At this point we need to make a distinction for the two cases: $m$ is odd and even. Clifford multiplication gives the following isomorphisms:
\begin{align}
\mathrm{cl}: i \Lambda^2 \oplus \Lambda^4 \oplus \cdots \oplus s_{m} \Lambda^{m}_+ \to i\mathfrak{su}(S_+) \hspace{5 ex} \text{if } m\text{ is even}
\label{4m-Clifford-isomorphism}  \\
\mathrm{cl}: i \Lambda^2 \oplus \Lambda^4 \oplus \cdots \oplus s_{m-1} \Lambda^{m-1} \to i\mathfrak{su}(S_+) \hspace{5 ex} \text{if } m\text{ is odd}\label{4m-2-Clifford-isomorphism}
\end{align}
where $\Lambda^{m}_+$ is the $+1$ eigenspace of $*$ acting on $\Lambda^{m}$ when $m$ is even. Given $\phi \in S_+$, we write $q(\phi)$ for the form which corresponds under~\eqref{4m-Clifford-isomorphism} or~\eqref{4m-2-Clifford-isomorphism} to $E_\phi \in i\mathfrak{su}(S_+)$.
Define $\Pi$ to be the projection:
\begin{align*}
    \Pi: \Omega^{\mathrm{even}}\otimes\mathbb{C}\rightarrow i\mathrm{Re}(\Omega^0)\omega\,\bigoplus_{p>0} s_{2p}\mathrm{Re}(\Omega^{0,2p})
\end{align*}
We consider equations for $(A,\beta,\phi)$ where $A\in\mathcal{A},\beta=\beta_3+\beta_5+\cdots+\beta_{(2\lceil\frac{m}{2}\rceil-1)}$ with $\beta_k\in$ Re$(\Omega^{0,k})$ and $\phi\in\Gamma(S_+)$. We set
\begin{align}
    \tilde{\beta}=\begin{cases}
	\displaystyle\sum_{k=1}^{(\frac{m}{2}-1)}\big(s_{2k+1}\beta_{2k+1}+s_{2m-2k-1}*\beta_{2k+1}\big) & \text{if } m\text{ is even}\\
	\displaystyle\sum_{k=1}^{(\lfloor\frac{m}{2}\rfloor-1)}s_{2k+1}\beta_{2k+1}+\displaystyle\sum_{k=1}^{\lfloor\frac{m}{2}\rfloor}s_{2m-2k-1}*\beta_{2k+1}& \text{if } m\text{ is odd} 	
	\end{cases}
\end{align}
and \begin{align}\label{Dirac 0}
    D_{A,\beta}=D_A+\mathrm{cl}(\tilde\beta)
\end{align}
This is a self-adjoint operator on spinors, which swaps chirality, $ D_{A,\beta}:\Gamma(S_+)\rightarrow\Gamma(S_-)$. Write
\begin{align}\label{curvature 0}
    F_{\beta}=\begin{cases}2\displaystyle\sum_{k=1}^{(\frac{m}{2}-1)}\big(s_{2k+2}d\beta_{2k+1}+(-1)^{\frac{m}{2}+k+1}s_{2k}d^*\beta_{2k+1}\big)&\text{if } m\text{ is even}\\
    2\displaystyle\sum_{k=1}^{(\lfloor\frac{m}{2}\rfloor-1)}s_{2k+2}d\beta_{2k+1}+(-1)^{(\lfloor\frac{m}{2}\rfloor+1)}2\displaystyle\sum_{k=1}^{\lfloor\frac{m}{2}\rfloor}s_{2k}d^*\beta_{2k+1}& \text{if } m\text{ is odd} 
    \end{cases}
\end{align}
The point of these equations is that as proved in \cite{Fine}, the Dirac operator $D_{A,\beta}$ has a Weitzenb\"ock formula of the form
\begin{align}\label{Weitzenbock}
    D_{A,\beta}^*D_{A,\beta}=\nabla_{A,\beta}^*\nabla_{A,\beta}+\frac{s}{4}+\frac{1}{2}\mathrm{cl}(F_A+F_\beta)+Q(\beta)
\end{align}
where $\nabla_{A,\beta}$ is a unitary connection on $S$ defined by $A$ and $\beta$ and $Q(\beta)$ is a zeroth order term which is purely algebraic in $\beta$. So, \eqref{curvature 0} prescribes the principal part of the Weitzenb\"ock remainder. 
\begin{definition}
We define $\psi_0\in\Omega^0$ to be the constant section of $M\times\mathbb{C}\subset S^{\mathrm{can}}_+$ which assigns $1\in\mathbb{C}$ to every point of $M$. 
\end{definition}
\begin{comment}
\begin{definition}\label{constant}
For $m\in\mathbb N, m>1$, we define a constant $c_m$ by the following equation
\begin{align*}
    \frac{1}{2^{m-1}}\int_M q(\psi_0)\wedge \omega^{m-1}=ic_m\int_M \omega^m
\end{align*}
\end{definition}
We calculate this constant for $m=2,3,4,5$ in \S. $c(2)=c(3)=c(4)=1, c(5)=\frac{1}{2}$. We are now ready to define the new gauge theoretic equations.
\end{comment}
\begin{definition}
The equations on an almost Hermitian manifold $(M,g,J)$ for 
$A\in\mathcal{A},\beta\in\displaystyle\bigoplus_{p>0}\mathrm{Re}(\Omega^{0,2p+1}),\phi\in\Gamma(S_+)$ are
\begin{align}
    D_{A,\beta}\phi=0\label{Dirac}\\
    \Pi(F_A-F_{{A^{\mathrm{Ch}}}}+F_\beta+rq(\psi_0))=\Pi(q(\phi))\label{Curvature}
\end{align}
\end{definition}
As mentioned before, $F_{{A^{\mathrm{Ch}}}}$ denotes the curvature of the anticanonical line bundle $K^{-1}$ with respect to the \textit{Chern} connection ${A^{\mathrm{Ch}}}$. The gauge group $\mathcal{G}=$ Map$(M,S^1)$ acts on $(A,\phi)$ by pull back and trivially on $\beta$. Notice when $n=2m=4$, $\beta=0$ and $\Pi(F_A-F_{{A^{\mathrm{Ch}}}}+rq(\psi_0))=F_A^+-F_{{A^{\mathrm{Ch}}}}^++\frac{ir}{4}\omega$ and $\Pi(q(\phi))=q(\phi)$ and hence the equations \eqref{Dirac}, \eqref{Curvature} indeed become the four-dimensional Seiberg--Witten equations \eqref{Dirac 4d}, \eqref{Curvature 4d}.\par 
One might wonder what is the significance of the parameter $r$ in the curvature equation. Why not just consider $r=1$ for example? The reason again stems from Taubes' analysis. For the original Seiberg--Witten equations $D_A\phi=0, F_A^+=q(\phi)$ on a symplectic manifold, if one perturbs the curvature equation by $i\omega$ on the left hand side, where $\omega$ is the symplectic form and then do the standard ``\textit{zooming in}'' procedure, i.e., rescale the metric in a geodesic ball around a point by a constant and consider the equations in the rescaled metric, the constant naturally appears as a multiple of the perturbation term $i\omega$. As we move closer and closer to the point, the rescaled metric converges to the flat metric on $\mathbb R^4$ which parallels taking this constant to infinity, i.e. taking \textit{Taubes' limit}. It is natural to ask whether a similar Taubes' limit ($r\rightarrow\infty$) persists for the pair of equations above in dimensions $n\geq 6$. We discuss this in detail in \S\S\ref{limit}.
\begin{remark}
    Similar gauge theoretic equations extending the Seiberg--Witten equations in higher dimensions appear in other articles. Most noteworthy mentions are Tanaka's preprint \cite{Tanaka} and the author's work with Fine \cite{Fine} (see also \cite{Partha}). In dimension $6$, the equations presented here are very similar to Tanaka's equations.\par
    In \cite{Fine}, the proposed equations can be defined on any spin$^\mathbb C$-manifold of any dimension $>2$. The equations proposed in this article can only be defined on manifolds with an almost complex structure. However, for the set of equations presented in this article, the canonical spin$^\mathbb C$-structure induced by the almost complex structure gives us an index $0$ problem to work with, which is not necessarily the case in \cite{Fine}. So, in principle one can hope to count number of solutions and get an invariant in the presence of a compactness result.\par
    One other major difference is when $m$ is odd, the recipe in \cite{Fine} requires two spinors and accordingly two Dirac and two curvature equations to form an elliptic system; whereas in our recipe we would only need one spinor and hence one Dirac and one curvature equation. 
\end{remark}
\begin{remark}
    The assumption that the base manifold $M$ admits an almost complex structure $J$ is not as big as it appears. For the four-dimensional Seiberg--Witten equations, there is no known example of a simply connected oriented closed four-dimensional manifold $M$ with $b^2_+>1$ and a spin$^\mathbb C$-structure $\mathfrak s$, such that the Seiberg--Witten invariant $SW_M(\mathfrak s)\neq 0$ and $\mathfrak s$ does not come from an almost complex structure. In dimension six, a manifold is spin$^\mathbb C$ iff it has an almost complex structure; see~\S\S\ref{spin geometry 6} for a proof. Conditions for the existence of almost complex structures in dimensions eight and ten are discussed in \cite{Heaps}.
\end{remark}
\begin{theorem}\label{index}
Let $(M,g,J)$ be a closed almost Hermitian manifold. The equations \eqref{Dirac} and \eqref{Curvature} are elliptic modulo gauge, with index
\begin{align*}
    2\int_M ch(\mathcal L)\wedge td(TM,J)-2\int_M td(TM,J) 
\end{align*}
\end{theorem}
In particular if $\mathcal L$ is the trivial complex line bundle, the index is $0$.\par
The algebraic term $Q(\beta)$ in the Weitzenb\"ock formula \eqref{Weitzenbock} turns out to be quite important in understanding the potential compactness or non-compactness of the moduli space of the solutions \eqref{Dirac}, \eqref{Curvature}. In dimension $n=4,\beta=0$ and therefore $Q(\beta)=0$.  For $n=6,8$,  formulae for $Q(\beta)$ appear in \cite{Fine}.
\begin{itemize}
    \item For $n=6,$ 
    \begin{align}
        Q(\beta)=-2|\beta|^2\label{Weitzenbock 6}
    \end{align}
    \item For $n=8,$
    \begin{align}
        Q(\beta)=-2\mathrm{cl}(\beta)^2-4|\beta|^2\label{Weitzenbock 8}
    \end{align}
\end{itemize}
\begin{theorem}\label{theorem1 symplectic}
Let $(M,g,J)$ be an almost K\"ahler (symplectic) manifold. For any $r>0$ there exists a canonical solution to the equations \eqref{Dirac}, \eqref{Curvature} for the canonical spin$^{\mathbb C}$ structure. We call it the ``canonical'' solution. For $n=4$, we recover Taubes' solution \cite{Taubes1} to the four-dimensional Seiberg--Witten equations \eqref{Dirac 4d}, \eqref{Curvature 4d} for the canonical spin$^{\mathbb C}$ structure.
\end{theorem}
\begin{theorem}\label{transversal}
Let $(M,g,J)$ be a closed almost K\"ahler (symplectic) manifold. For large enough $r>0$, the canonical solution is transversely cut out.
\end{theorem}
Taubes' \cite{Taubes1} obstruction result for symplectic structures on $4$-manifold hinges on the fact that for large enough $r$, the canonical solution is the unique solution to the equations \eqref{Dirac 4d}, \eqref{Curvature 4d} and it is cut out transversely. We do not quite have such a strong result, however we prove the following in dimension $6$.
\begin{theorem}\label{theorem2 symplectic}
    Let $(M,g,J)$ be a closed almost K\"ahler manifold of dimension $n=6$. For the canonical spin$^\mathbb C$-structure, if the three form $\beta$ is zero, then for large enough $r$, modulo gauge the canonical solution is the unique one. 
\end{theorem}
\begin{comment}
    \begin{theorem}\label{transversal}
    Let $(M,g,J)$ be a closed almost K\"ahler manifold of dimension $n=6$. For large enough $r$, the canonical solution described in theorem \ref{theorem1 symplectic} is transversely cut out.
\end{theorem}
\end{comment}
The Seiberg--Witten equations \eqref{Dirac 4d}, \eqref{Curvature 4d} exhibit notable behavior when applied to Kähler surfaces, specifically in relation to vortex equations. 
Let $(M,g,J)$ be a closed K\"ahler manifold of complex dimension $m$, and let $\mathcal L\to M$ be a Hermitian line bundle. For $\tau\in\mathbb{R}$, the $\tau$-vortex equations for a unitary connection $B$ on $\mathcal L$ and a section $\phi\in\Omega^0(M,\mathcal L)$ are \begin{align} \bar\partial_B\phi &=0\\ i\Lambda F_B &=\frac12\bigl(\tau-|\phi|^2\bigr)\\ F_B^{0,2}&=0 \end{align} Here $\Lambda$ denotes the $L^2$-adjoint of wedging with the K\"ahler form $\omega$. The vortex moduli space has a natural map to the Picard variety parametrising holomorphic structures on the underlying smooth line bundle $\mathcal L$. Over a holomorphic line bundle $\mathcal L$ admitting non-zero holomorphic sections, the fibre is \[ \mathbb{P}H^0(M,\mathcal L)\] A vortex pair $(B,\phi)$ determines the effective divisor \[ D=\phi^{-1}(0) \] whose Poincar\'e dual is $c_1(L)$. Conversely, an effective divisor $D$ with \[ \operatorname{PD}[D]=c_1(\mathcal L) \] determines a holomorphic line bundle together with a holomorphic section vanishing precisely along $D$. Its degree is \[ \deg(\mathcal L) = \int_M c_1(\mathcal L)\wedge\frac{\omega^{m-1}}{(m-1)!}\] For \[ \tau>\frac{4\pi\,\deg(\mathcal L)}{\operatorname{vol}(M)}\] the vortex moduli space is identified with the corresponding space of effective divisors representing $c_1(\mathcal L)$; see \cite{Brad}.\par
On closed K\"ahler surfaces, the Seiberg--Witten equations can be reduced, under cer
tain conditions, to the vortex equations. This reduction was shown by Witten \cite{EW} as part of a broader connection between Seiberg--Witten theory and symplectic geometry. For example, in symplectic $4$-manifolds, Taubes established a correspondence where Seiberg--Witten solutions correspond to pseudoholomorphic curves \cite{Taubes4}, which are governed by the vortex-like equations under certain limits. This insight has allowed the Seiberg--Witten equations to be applied in tackling problems such as the Thom conjecture \cite{KM}. Vortices also appear as solutions of the equations \eqref{Dirac}, \eqref{Curvature} in higher dimensions.
\begin{theorem}\label{vortex}
Let $(M,g,J)$ be a K\"ahler manifold. For $B$ a unitary connection on $\mathcal L,$ take $A={A^\mathrm{Ch}}+2B,\beta=0$ and $\phi=\varphi\in\Omega^0(M,\mathcal L)\subset \Gamma(S_+)$. Then equations \eqref{Dirac}, \eqref{Curvature} reduce to vortex-like equations
\begin{align}
    \bar\partial_B \varphi=0\label{red1}\\
    i\Lambda F_B=\frac{c_m}{2}(r-|\varphi|^2)\label{red2}\\
    F_B^{0,2}=0\label{red3}
\end{align}
where $c_m$ is a positive constant depending only on the dimension $2m=n$.
\end{theorem}
Moreover if we make the identification: $\tilde\varphi=\sqrt{c_m}\varphi,\tau=c_m r$, the equations above read exactly as the vortex equations:
\begin{align*}
    \bar\partial_B \tilde\varphi=0\\
    i\Lambda F_B=\frac{1}{2}(\tau-|\tilde\varphi|^2)\\
    F_B^{0,2}=0
\end{align*}
Therefore equations \eqref{red1}, \eqref{red2}, \eqref{red3} have the same moduli space as the vortex equations for $r>\frac{4\pi\mathrm{deg}(\mathcal L)}{c_m\mathrm{vol}(M)}$. \par 
One of the major reasons behind the success of the Seiberg--Witten theory is the compactness of the moduli space of the equations. The heart of the compactness argument stems from the standard Weitzenb\"ock formula:
\begin{align*}
    D_A^2=\nabla_{A}^*\nabla_A+\frac{s}{4}+\frac{1}{2}\mathrm{cl}(F_A)
\end{align*}
and the maximum principle. However in dimension higher than $4$, we have this extra parameter $\beta\in\oplus_{p>0}\mathrm{Re}(\Omega^{0,2p+1})$ and as explained before, the corresponding Weitzenb\"ock formula \eqref{Weitzenbock} has a quadratic term $Q(\beta)$ which is troublesome while trying to get a uniform $C^0$-limit of the spinor mimicking the $4$-dimensional Seiberg--Witten theory. Nevertheless we prove some a priori estimates for the variables in~\S\S\ref{estimates}.\par
Donaldson \cite{Donaldson} realised that there exists a variational description of the Seiberg--Witten equations. In dimension $4$, on a compact manifold with a spin$^\mathbb C$ structure, solutions $(A,\phi)$ to the unperturbed Seiberg--Witten equations: $D_A\phi=0,F_A^{+}=q(\phi)$ are absolute minima of the following energy functional:
\begin{align*}
    \mathcal E(A,\phi)=\int \Big(|\nabla_A\phi|^2+\frac{1}{2}|F_A|^2+\frac{1}{8}(|\phi|^2+s)^2\Big)dv
\end{align*}
An analogous energy functional also exists for the Seiberg--Witten equations with Taubes' perturbation \eqref{Dirac 4d}, \eqref{Curvature 4d} on a closed almost complex $4$-manifold. We prove a similar result for solutions $(A,\beta,\phi)$ to the equations \eqref{Dirac}, \eqref{Curvature} on a closed K\"ahler $3$-fold. Similar results hold in all dimensions, with different constants appearing in various places; we work in dimension $6$ to keep the discussion concrete.
\begin{definition}
Let $(M,g,J)$ be a closed K\"ahler $3$-fold. For $S_+=\oplus_p\Omega^{0,2p}(\mathcal{L}), A\in\mathcal{A},$ one can write $A=A^{\mathrm{Ch}}+2B$ with abuse of notation. Here $A^{\mathrm{Ch}}$ is the Chern connection on the $K^{-1}$ and $B$ is a unitary connection on $\mathcal{L}$. Given such an $A=A^{\mathrm{Ch}}+2B,\beta\in$ Re$(\Omega^{0,3}),\phi\in\Gamma(S_+)$, we define the energy of $(A,\beta,\phi)$ to be 
\begin{align*}
    &\mathcal E(A,\beta,\phi)\\
    &:=\int\Big(|\nabla_{A,\beta}\phi|^2+\frac{s}{4}|\phi|^2+2|F_B|^2+4|d^{*}\beta|^2+\frac{3r^2}{16}+\frac{1}{48}(3|\phi^0|^2-|\phi^{0,2}|^2)^2+\frac{1}{2}|\phi^0|^2|\phi^{0,2}|^2\\
    &\hspace{5 ex}-\frac{2}{3}|\Lambda(F_B)|^2-2|\beta|^2|\phi|^2-\frac{r}{8}(3|\phi^0|^2-|\phi^{0,2}|^2)+\big\langle 4(F_B)^{1,1}_p+2F_{A^\mathrm{Ch}},q(\phi)\big\rangle\Big)dv\\
    &\hspace{5 ex}-2\pi r\mathrm{deg}(\mathcal L)
\end{align*}
\end{definition}
\begin{theorem}\label{Energy}
    In dimension $6$, on a closed K\"ahler $3$-fold, for any $(A,\beta,\phi)$, we have the inequality
    \begin{align*}
        \mathcal E(A,\beta,\phi)\geq -8\pi^2\langle c_1(\mathcal L)^2\cup [\omega],[M]\rangle
    \end{align*}
    with equality if and only if $(A,\beta,\phi)$ solves the equations \eqref{Dirac}, \eqref{Curvature}.
\end{theorem}
This article is structured as follows. In \S\ref{s2}, we prove the ellipticity of the equations \eqref{Dirac}, \eqref{Curvature} modulo gauge and compute their index, establishing theorem \ref{index}. \S \ref{soln} is devoted to solutions of the equations: we first review the Dirac operator and Clifford multiplication in almost Kähler geometry, then construct the canonical solution and prove its transversality for large \(r\), proving theorems \ref{theorem1 symplectic} and \ref{transversal}. We also study the K\"ahler case, where the equations reduce to the vortex equations, proving theorem \ref{vortex}, and discuss the higher-dimensional analogue of Taubes' limiting procedure. In \S \ref{sec4}, we specialise to dimension six. After discussing some relevant spin geometry in dimension six, we derive a priori estimates and prove the uniqueness of the canonical solution under the assumption \(\beta=0\), establishing theorem \ref{theorem2 symplectic}. Finally, we derive the energy identity and prove theorem \ref{Energy}.\par
\vspace{2 ex}
\textbf{Acknowledgments.} The author wishes to thank Joel Fine, for his encouragement and for various fruitful discussions. This work was supported by European Union’s Horizon 2020 Europe research and innovation programme under the Marie Sklodowska-Curie grant agreement no 101034255.
\section{Ellipticity and the index}\label{s2}
\begin{proof}[Proof of theorem \ref{index}]
Ignoring the zeroth order terms, the linearisation of equations \eqref{Dirac}, \eqref{Curvature} (modulo gauge) at a solution $(A,\beta,\phi)$ lead us to the following differential operator
\begin{align}
    \Gamma(S_+)\bigoplus i\Omega^1\bigoplus_{p>0}\mathrm{Re}(\Omega^{0,2p+1})&\rightarrow \Gamma(S_-)\bigoplus i\mathrm{Re}(\Omega^0)\bigoplus i\mathrm{Re}(\Omega^0)\omega\,\bigoplus_{p>0} \mathrm{Re}(\Omega^{0,2p})\\
   \big(\varphi,a,b\big)&\mapsto \big(D_A\varphi,d^*a,\Pi(da+F_b)\big)\label{1}
\end{align}
The next step is to identify $\Omega^0$ as Re$(\Omega^0)\oplus$ Re$(\Omega^0)\omega$ and realise that up to some signs, factors of $i$, constants and some zeroth order terms, the operator above is secretly nothing but
\begin{align}
    \Gamma(S_+)\bigoplus\Omega^{0,\mathrm{odd}}&\rightarrow \Gamma(S_-)\bigoplus \Omega^{0,\mathrm{even}}\\
   \big(\varphi,\alpha\big)&\mapsto \big(D_A\varphi, (\bar\partial+\bar\partial^*)\alpha\big)\label{2}
\end{align}    
More concretely speaking the ellipticity of the operator \eqref{1} is equivalent to the ellipticity of the operator \eqref{2}. The explicit calculations which make the symbols of these two operators equivalent are a little tedious, for the details please see proof of theorem \ref{transversal} in \S\S\ref{ansatz}. Operator \eqref{2} is elliptic simply because it comprises the Dirac and the Dolbeault operator. Moreover their indices are also the same and hence the theorem follows from the Atiyah--Singer index theorem \cite{index}.
\end{proof}
\begin{comment}
\subsection{Weitzenb\"ock formulae}\label{Weitzenbock}
The main tool behind proving the compactness for the four dimensional Seiberg--Witten moduli space is the Weitzenb\"ock formula for the Dirac operator $D_A$:
\begin{align*}
    D_A^*D_A=\nabla_A^*\nabla_A+\frac{s}{4}+\frac{1}{2}\mathrm{cl}(F_A)
\end{align*}
In the higher dimensional gauge theoretic equations \eqref{Dirac}, \eqref{Curvature} we have a perturbed version of the Dirac opertaor: $D_{A,\beta}$, which is still self-adjoint. We prove the corresponding Weitzenb\"ock formulae in dimensions $n=6,8$. Explicit formulae are stated in \cite{Fine}, where the details of the computations of the quadratic algebraic part are omitted. We provide these details here, following the notation of \cite{Fine}.
\begin{theorem}
For $n=6$,
\begin{align}
   D_{A,\beta}^*D_{A,\beta}=\nabla_{A,\beta}^*\nabla_{A,\beta}+\frac{s}{4}+\frac{1}{2}\mathrm{cl}(F_A+F_\beta)-2|\beta|^2
\end{align}
\end{theorem}
\begin{proof}

\end{proof}
\begin{theorem}
For $n=8$,
\begin{align}
   D_{A,\beta}^*D_{A,\beta}=\nabla_{A,\beta}^*\nabla_{A,\beta}+\frac{s}{4}+\frac{1}{2}\mathrm{cl}(F_A+F_\beta)-2\mathrm{cl}(\beta)^2-4|\beta|^2
\end{align}
\end{theorem}
\end{comment}
\section{Solutions to the equations}\label{soln}
\subsection{Dirac operator in almost K\"ahler geometry}
For any almost Hermitian manifold $(M,g,J)$ of dimension $n=2m\geq 4$, a connection $\nabla$ on $M$ (acting on sections of the tangent bundle $TM$) is \textit{Hermitian} if it preserves the metric $g$ and the almost complex structure $J$:
\begin{align*}
    \nabla g=0,\, \nabla J=0
\end{align*}
Each connection $\nabla$ determines a \textit{Cauchy Riemann operator}, denoted by $\partial^\nabla$, defined as the $(0,1)$-part of $\nabla$:
\begin{align*}
    \bar\partial^\nabla _X Y=\frac{1}{2}(\nabla _X Y+J \nabla_{JX} Y)
\end{align*}
There is also an \textit{intrinsic} Cauchy Riemann operator coming from the almost complex structure $J$ defined by
\begin{align*}
    \bar\partial_X Y=\frac{1}{4}([X,Y]+[JX,JY]+J[JX,Y]-J[X,JY])
\end{align*}
Gauduchon \cite{Paul} distinguished a real affine line in the set of Hermitian connections. We pick one element from the affine line, the \textit{Chern connection}: $\nabla^{\text{Ch}}$. The terminology stems from the fact that, in the integrable case, it coincides with the Chern connection of the tangent bundle viewed as a Hermitian, holomorphic bundle of rank $m$. The Chern connection is characterised by the fact that its Cauchy Riemann operator coincides with the intrinsic one:
\begin{align*}
    \bar\partial^{\nabla^{\text{Ch}}}= \bar\partial
\end{align*}
The metric $g$ defines a Hermitian metric on the anticanonical bundle $K^{-1}$ and the Chern connection induces a unitary connection on $K^{-1}$, which we call by ${A^{\text{Ch}}}$.\par 
To define a spin connection on $S=S^{\mathrm{can}}\otimes\mathcal{L}=S_+\oplus S_-$, we need a unitary connection $A$ on $L(S)=K^{-1}\otimes\mathcal{L}^2$. We can write any such unitary connection as $A=A^{\text{Ch}}+2B$, $B$ being a connection on $\mathcal L$. Gauduchon \cite{Paul} proved that on an almost K\"ahler manifold $(M,g,J)$ (i.e., $d\omega=0$), the formula for the Dirac operator is the following.
\begin{align}
    D_A=\sqrt{2}(\bar\partial_B+\bar\partial^*_B)
\end{align}
In dimension $4,$ Taubes \cite{Taubes1} obtained the same result by indirect arguments and it is a key fact in his \textit{magnum opus}: ``SW$=$Gr''.\par 
For any unitary connection $A$ on $L(S)$ and a spinor $\psi\in \Gamma(S)$, the gauge group $\mathcal{G}=$ Map$(M,S^1)$ acts on the pair $(A,\psi)$ by pull-back. Hence for $g\in\mathcal{G},$
\begin{align*}
    g\cdot (A,\psi)=(A-2g^{-1}dg, g\psi)
\end{align*}
Moreover, $(A,\psi)$ solves the Dirac equation $D_A\psi=0$ iff $g\cdot (A,\psi)$ solves it.\par
In particular when $\mathcal L$ is the trivial line bundle, i.e., we work with the canonical spin$^{\mathbb C}$ structure, we have
\begin{align}\label{Dirac1}
    D_{A^\text{Ch}}= \sqrt{2}(\bar\partial+\bar\partial^*)
\end{align}
$\bar\partial^*$ is the Hermitian adjoint of $\bar\partial$.\par 
The discussion above along with the identity \eqref{Dirac1} gives us the following lemma.
\begin{lemma}\label{lemma 100}
    Up to gauge equivalence there exists a canonical connection $A_0$ and a non-trivial spinor $\varphi_0\in\Omega^0 \subset \Gamma(S^{\mathrm{can}}_+)$, such that $D_{A_0}\varphi_0=0$. $\varphi_0$ can be taken to have pointwise norm one.
\end{lemma}
\begin{proof}
    $A_0=A^{\text{Ch}},\varphi_0=\psi_0$ solves the Dirac equation $D_{A_0}\varphi_0=0$ because $d\psi_0=0$ and therefore $\bar\partial\psi_0=0$.
\end{proof}
In spinorial terms the Chern connection is characterised as follows (see \S\S 1.4.3. in \cite{Nicol} for a detailed exposition). For the constant section $\psi_0,$ $A^{\text{Ch}}$ is the unique connection on $K^{-1}$ such that the induced connection on $S,\nabla_{A^\text{Ch}}$ satisfies 
\begin{align*}
    \nabla_{A^\text{Ch}}\, \psi_0\in\Gamma(\Lambda^1\otimes \Lambda^{0,2})
\end{align*}
In fact just using this condition above one can prove that $D_{A^{\text{Ch}}}\,\psi_0=0$. To see this first notice that $\Omega^0$ is a $-mi$ eigenspace of the Clifford action of the symplectic form : cl$(\omega)$ (see lemma \ref{lemma 1} for a proof) and $\Omega^{0,2}$ is a $ci$ eigenspace of cl$(\omega)$, where $c\neq -m$. For example, $c=2$ when $m=2$, $c=1$ when $m=3$, and $c=0$ when $m=4$ \cite{Partha}. Now apply $D_{A^{\text{Ch}}}$ on cl$(\omega)\psi_0$ to get 
\begin{align}\label{eq 1}
    D_{A^{\text{Ch}}}\big(\text{cl}(\omega)\psi_0\big)= \text{cl}\big(d\omega+d^*\omega\big)\psi_0+\text{cl}\big((1\otimes \text{cl}(\omega))\,\nabla_{A^{\text{Ch}}} \psi_0\big)
\end{align}
The rightmost term here means the following: start with $\nabla_{A^{\text{Ch}}} \psi_0\in\Gamma(\Lambda^1\otimes\Lambda^{0,2})$, first apply $1\otimes \text{cl}(\omega)$ to get another element of $\Gamma(\Lambda^1\otimes\Lambda^{0,2})$ and then apply cl to get an element of $\Gamma(S_-)$. Because $\omega$ is closed and co-closed, equation \eqref{eq 1} becomes
\begin{align*}
    -mi D_{A^{\text{Ch}}}(\psi_0)=ci D_{A^{\text{Ch}}}(\psi_0)
\end{align*}
Since $c\neq -m$, this proves $D_{A^{\text{Ch}}}\,\psi_0=0$.\par
$\nabla_{A^{\text{Ch}}}\,\psi_0=0$ iff $(M,g,J)$ is K\"ahler, instead if it is almost K\"ahler, we still have $D_{A^{\text{Ch}}}\,\psi_0=0$.
\subsection{Clifford multiplication in almost complex geometry}
Let $\alpha\in\Omega^1\otimes\mathbb{C}$ be a complexified one-form and $\xi\in\Omega^{0,k}$ be a spinor. Clifford action of $\alpha=(\alpha^{0,1}+\alpha^{1,0})$ on $\xi$ is given by the following formula:
\begin{align}\label{Clifford one form}
c(\alpha)\xi=\sqrt{2}(\alpha^{0,1}\wedge\xi-\overline{\alpha^{1,0}}\righthalfcup\xi) 
\end{align}
$\overline{\alpha^{1,0}}\righthalfcup \xi\in\Omega^{0,k-1}$ is the contraction of $\alpha$ and $\xi$ defined by $\overline{\alpha^{1,0}}\righthalfcup \xi:=\iota_{(\overline{\alpha^{1,0}})^{\#}}(\xi)$, where $(\overline{\alpha^{1,0}})^{\#}$ is the vector metric dual to $\overline{\alpha^{1,0}}$. At a point $p\in X,$ if $\xi=e_1\wedge\cdots\wedge e_k$
\begin{align*}
     \overline{\alpha^{1,0}}\righthalfcup(e_1\wedge\cdots\wedge e_k):=\sum_{i=1}^k (-1)^{i-1}\langle e_i,\overline{\alpha^{1,0}}\rangle e_1\wedge\cdots\wedge\cdots\wedge\hat{e_i}\wedge\cdots\wedge e_k
\end{align*}
\begin{remark}\label{remark Clifford action}
    Notice $c(\alpha)$ takes positive spinors: $\oplus_p\Omega^{0,2p}$ to negative spinors: $\oplus_p\Omega^{0,2p+1}$ and vice versa. Moreover it says that for a $(p,q)$ form $\kappa, c(\kappa)$ takes a spinor in $\Omega^{0,k}$ to $\Omega^{0,k+q-p}.$ Hence the action is trivial if $k+q-p<0$ or $k+q-p>m$. In particular the Clifford action of forms in $\Omega^{k,0}$ on spinors in $\Omega^0\subset \Gamma(S_+)$ is trivial for $k>0$.
\end{remark}
\begin{lemma}\label{lemma 1}
    $\Omega^0$ is a $-mi$ eigenspace of the Clifford action of the symplectic form : $\mathrm{cl}(\omega)$.
\end{lemma}
\begin{proof} Choose local coordinates $x_1,y_1=Jx_1,\dots, x_m,y_m=Jx_m$ at a point $x\in M$ such that the symplectic form $\omega$ at $x$ reads
\begin{align*}
    \omega=\sum_{j=1}^m dx_j\wedge dy_j
\end{align*} For a spinor $\phi^0\in\Omega^0$ we get
    \begin{align*}
        \mathrm{cl}(dx_j\wedge dy_j)\phi^0&=\mathrm{cl}(dx_j)\circ \mathrm{cl}(dy_j)\phi^0\\
        &=\mathrm{cl}(dx_j)(\frac{i}{\sqrt{2}}d\bar z_j\wedge\phi^0)\hspace{1.5 ex}[dz_j=dx_j+idy_j]\\
        &=-i\phi^0\\
    \text{Hence, }\mathrm{cl}\big(\sum_{j=1}^m dx_j\wedge dy_j\big)\phi^0&=-mi\phi^0\tag*{\qedhere}
    \end{align*}
\end{proof}
\begin{lemma}\label{lemma 2}
    For any $\beta^{0,k}\in\Omega^{0,k}$ and $\phi^0\in\Omega^0$,
    \begin{align*}
        \mathrm{cl}(\beta^{0,k})\phi^0=(\sqrt{2})^k\phi^0\beta^{0,k}
    \end{align*}
\end{lemma}
\begin{proof}
    Choose local coordinates $x_1,y_1=Jx_1,\dots, x_m,y_m=Jx_m$ at a point $x\in M$. Define $dz_j=dx_j+idy_j$. Using the identity
    \begin{align*}
        \mathrm{cl}(v\wedge w^k)=\mathrm{cl}(v)\circ \mathrm{cl}(w^k)+\mathrm{cl}(v\righthalfcup w^k)\hspace{2 ex} \text{for }v\in\Omega^1,w\in\Omega^k
    \end{align*} and formula \ref{Clifford one form} for Clifford multiplication, we deduce that for any pairwise distinct $j_1,j_2,\cdots,j_k$, we have
    \begin{align*}
        \mathrm{cl}(d\bar z_{j_1}\wedge d\bar z_{j_2}\wedge \cdots\wedge d\bar z_{j_k})\phi^0&=\mathrm{cl}(d\bar z_{j_1})\circ \cdots \circ \mathrm{cl}(d\bar z_{j_k})\phi^0\\&=(\sqrt{2})^k\phi^0\, d\bar z_{j_1}\wedge d\bar z_{j_2}\wedge \cdots\wedge d\bar z_{j_k}
    \end{align*}
    Since Clifford multiplication is linear in $\beta^{0,k}$, the lemma follows.
\end{proof}
\subsection{Ansatz on symplectic and K\"ahler manifolds}\label{ansatz}\label{ansatz}
\subsubsection*{Canonical solution on almost K\"ahler (symplectic) manifolds}\label{symplectic}
\begin{proof}[Proof of theorem \ref{theorem1 symplectic}]
    We define the canonical solution as 
    \begin{align}\label{canonical}
        A=A^{\mathrm{Ch}},\,\beta=0,\,\phi=\sqrt{r}\psi_0
    \end{align}
Lemma \ref{lemma 100} tells us that the canonical solution solves both the Dirac and the curvature equations \eqref{Dirac}, \eqref{Curvature}.
\end{proof}
\begin{proof}[Proof of theorem \ref{transversal}]
The deformation problem for the canonical spin$^\mathbb C$ structure has Fredholm index zero, therefore proving injectivity for the linearisation will then imply surjectivity and hence transversality. The proof given here is inspired by the proof of transversality of the canonical solution in dimension $4$ given in Nicolaescu's book \cite{Nicol}. The infinitesimal variables in the linearised equations with appropriate gauge fixing lie in
\begin{align*}
    \Gamma(S_+^{\mathrm{can}})\bigoplus i\Omega^1\bigoplus_{p>1,p\text{ odd}} \mathrm{Re}(\Omega^{0,p})
\end{align*} 
The trick is to identify 
$$i\Omega^1\bigoplus_{p>1,p\text{ odd}} \mathrm{Re}(\Omega^{0,p})$$ with $$\Gamma(S_-^{\mathrm{can}})$$ and write both the linearised Dirac and the curvature equation (with gauge fixing) in terms of the Dirac operators. The difficult part of the proof involves this careful identification. The rest is quite straightforward. We give a sketch of the proof below.\par
Take $\lambda=\sqrt{r}$. First we show that the linearisation at the canonical solution gives us two Dirac equations with variables $\sigma\in\Gamma(S_+^{\mathrm{can}}),\eta\in\Gamma(S_-^{\mathrm{can}})$.
\begin{align}
    D_{A^{\mathrm{Ch}}}\sigma+\lambda G\eta=0\label{lin 1}\\
    (D_{A^{\mathrm{Ch}}}+R)\eta-\lambda H\sigma=0\label{lin 2}
\end{align}
where $R,G$ and $H$ are zeroth order real linear maps; moreover $G$ and $H$ are invertible. Substituting $\eta=-\frac{1}{\lambda} G^{-1}D_{A^{\mathrm{Ch}}}\sigma$ into equation \eqref{lin 2} we get
\begin{align}
    \big(H^{-1}(D_{A^{\mathrm{Ch}}}+R)G^{-1}D_{A^{\mathrm{Ch}}}\big)\sigma+\lambda^2\sigma=0\label{lin 3}
\end{align}
Next we will show that the operator $\mathcal A:=H^{-1}(D_{A^{\mathrm{Ch}}}+R)G^{-1}D_{A^{\mathrm{Ch}}}$ is strongly elliptic. Here we think of $S_+^{\mathrm{can}}$ as a real vector bundle equipped with the inner product
\begin{align*}
    \langle\varphi_1,\varphi_2\rangle_{\mathbb R}=\mathrm{Re}\langle \varphi_1,\varphi_2\rangle
\end{align*}
Strong ellipticity means that there exists a constant $\kappa>0$ (independent of $v^*$ and $\zeta$) such that
\begin{align}
     \mathrm{Re}\langle\mathrm{Sym}(\mathcal A)(v^*)\zeta,\zeta\rangle\geq \kappa |v^*|^2|\zeta|^2
\end{align}
for every $x\in M,$ every $v^*\in T^*_xM$ and every $\zeta\in (S_+^{\mathrm{can}})_x$. we write $
\mathrm{Sym}(\mathcal A)$ for the principal symbol of a differential operator $\mathcal A$. With this we use the G\aa rding's inequality, i.e., there exists positive constants $c$ and $C$ depending only on the geometry of $M$ such that 
\begin{align}
    \mathrm{Re}\langle\mathcal A\sigma,\sigma\rangle_{L^2}\geq c\|\sigma\|^2_{H^1}-C\|\sigma\|^2_{L^2}\label{Garding}
\end{align}
Now taking inner product with $\sigma$ on both sides of the equation \eqref{lin 3} and integrating over $M$ gives us
\begin{align*}
    0=\mathrm{Re}\langle\mathcal A\sigma,\sigma\rangle_{L^2}+\lambda^2\|\sigma\|^2_{L^2}
\end{align*}
With the inequality \eqref{Garding} this leads us to
\begin{align*}
    0\geq c\|\sigma\|^2_{H^1}+(\lambda^2-C)\|\sigma\|^2_{L^2}\
\end{align*}
Hence for large enough $\lambda^2> C, \sigma=0$. Equation \eqref{lin 1} then tells us $G\eta=0$ but since $G$ is invertible this tells us $\eta=0$. Together $\sigma=0,\eta=0$ will tell us that the linearisation has trivial kernel.\par 
The proof is quite long; we break it into three parts. The first part discusses the linearisation of the Dirac equation at the canonical solution deriving equation \eqref{lin 1}, the second part does the same for the curvature equation deriving equation \eqref{lin 2} and the last part shows that the operator $\mathcal A$ is strongly elliptic finishing the proof.\par 
\medskip
\noindent\textit{Part 1.}
Let $
(2ia,b,\sigma)
$
be an infinitesimal deformation of the canonical solution, where
\begin{align*}
a\in\Omega^1(M,\mathbb R),\qquad
b\in\bigoplus_{q\geq3,\ q\ {\rm odd}}\text{Re}(\Omega^{0,q}),
\qquad
\sigma\in\Gamma(S_+^{\text{can}})
\end{align*}
When $m$ is even the largest possible odd degree in $b$ is $m-1$;
when $m$ is odd the top degree is $m$. The linearisation of the Dirac equation \eqref{Dirac} is
\begin{equation}
D_{A^{\mathrm{Ch}}}\sigma
+
\lambda\,\mathrm{cl}(ia+\tilde b)\psi_0
=0
\label{eq:linD}
\end{equation}
Next we would like to rewrite equation \eqref{eq:linD} in a certain form. Before we do that, we need some preparation. For the real one-form $a$, define $
\alpha_1\in\Omega^{0,1}
$
by
\begin{equation}
a
=
\frac1{\sqrt2}
(\alpha_1+\overline{\alpha_1})
\label{eq:def-alpha1}
\end{equation}
Thus
$
a^{0,1}
=
\frac1{\sqrt2}\alpha_1.
$
Similarly for every odd $q\geq3$, define
$
\alpha_q\in\Omega^{0,q}
$
by
\begin{equation}
b_q
=
\frac1{(\sqrt2)^q}
\left(
\alpha_q+\overline{\alpha_q}
\right)
\label{eq:def-alphaq}
\end{equation}
We now calculate the Hodge-dual contribution using lemma~4.5 in \cite{Fine}.  In real dimension $2m$,
for $\gamma\in\Omega^k$,
\begin{equation}
\mathrm{cl}(*\gamma)
=
(-1)^{m+\frac{k(k+1)}{2}}i^m \mathrm{cl}(\gamma)\qquad\text{on } S_+
\label{eq:FG-star}
\end{equation}
Take $
q=2k+1.$
Then
\begin{align*}
\frac{q(q+1)}2
=
(2k+1)(k+1)
\equiv k+1\pmod2
\end{align*}
So equation \eqref{eq:FG-star} and lemma \ref{lemma 2} together give us
\begin{equation}
\mathrm{cl}(*b_q)\psi_0
=
(-1)^{m+k+1}i^m\alpha_q
\label{eq:star-bq}
\end{equation}
\begin{comment}
Recall
\begin{align*}
s_j=
\begin{cases}
1,&j\equiv0,3\pmod4,\\
i,&j\equiv1,2\pmod4.
\end{cases}
\end{align*}
In particular,
\begin{equation}
s_j^2=(-1)^{j(j+1)/2}.
\label{eq:s-square}
\end{equation}
\end{comment}
If $3\leq q<m$, the $q$-component of $\widetilde b$ is $
s_qb_q+s_{2m-q}*b_q.$ Again using lemma \ref{lemma 2} and equation \eqref{eq:star-bq} we get
\begin{equation}
c(\widetilde b_q)\psi_0
=
A_q\alpha_q,\,\,\text{where } A_q
=
s_q
+
(-1)^{m+k+1}i^m s_{2m-q}
\label{eq:Aq-action}
\end{equation}
When $m$ is odd and $q=m$, only the Hodge-dual term occurs in
$\widetilde b$.  In this case
\begin{equation}
A_m
=
(-1)^{m+\frac{m(m+1)}{2}}i^m s_m
\label{eq:Am-def}
\end{equation}
And finally we set
\begin{align}
A_1=i
\end{align} to have 
\begin{equation}
\mathrm{cl}(ia+\widetilde b)\psi_0
=
\sum_{\substack{1\leq q\leq m\\q\ {\rm odd}}}
A_q\alpha_q
\label{eq:odd-clifford-sum}
\end{equation}
For odd $q=2k+1$, define
\begin{equation}
\delta_q
=
\begin{cases}
(-1)^{\frac{m}{2}+k+1}
&\text{if }m \text{ is even}\\[1mm]
(-1)^{\frac{m+1}{2}}
&\text{if } m \text{ is odd}
\end{cases}
\label{eq:delta-def}
\end{equation}
With this notation, for $q<m$ the part of $F_b$ arising from $b_q$
is
\begin{equation}
2s_{q+1}\,db_q
+
2\delta_qs_{q-1}\,d^*b_q
\label{eq:Fbq}
\end{equation}
When $m$ is odd and $q=m$, only the second term occurs. Next we define real non-zero constants $r_q$ recursively by
\begin{equation}
r_1=1,
\qquad
r_q
=
\frac{\delta_q}{2}r_{q-2},
\qquad q\geq3,\quad q\ {\rm odd}
\label{eq:r-recursion}
\end{equation}
Next set 
\begin{equation}
\eta_q=r_q\alpha_q,
\qquad
\eta=\sum_{q\ {\rm odd}}\eta_q\in\Gamma(S_-^{\mathrm{can}})
\label{eq:eta-def}
\end{equation}
Equation \eqref{eq:linD} now reads
\begin{equation}
D_{A^{\mathrm{Ch}}}\sigma+\lambda G\eta=0,
\label{eq:first-system}
\end{equation}
where $G:S_-^{\mathrm{can}}\to S_-^{\mathrm{can}}$ preserves every degree and
\begin{equation}
G|_{\Omega^{0,q}}
=
G_q\,\text{Id},
\qquad
G_q=\frac{A_q}{r_q}
\label{eq:Gq-def}
\end{equation}
\begin{comment}
We shall prove below that every $A_q$, and hence every $G_q$, is
non-zero.
\end{comment}
The point of writing the linearised Dirac equation in this form is that in the linearised curvature equation, $\eta$ appears naturally. Therefore, it makes sense to keep it in the linearised Dirac equation as well to make our lives easier. Next we focus on the linearisation of the curvature equation.\par 
\medskip
\noindent\textit{Part 2.} The linearisation of the curvature equation \eqref{Curvature} at the canonical solution reads
\begin{align}
    \Pi(2ida+F_b)=\Pi(d_{\lambda\psi_0}q(\sigma))=\lambda \Pi(d_{\psi_0}q(\sigma))\label{linc}
\end{align}
The gauge group $
\mathcal G=C^\infty(M,S^1)$ acts on the parameter space by 
\begin{align*}
g\cdot(A,\beta,\phi)
=
\left(
A-2g^{-1}dg,\,
\beta,\,
g\phi
\right)
\end{align*}
Taking a curve in $\mathcal G, g_t=e^{itf}$, we can see that the infinitesimal gauge action at the canonical solution is given by
\begin{align*}
   i\mathrm{Re}(\Omega^0)&\longrightarrow i\mathrm{Re}(\Omega^1)\oplus \bigoplus_{q\geq3,\ q\ {\rm odd}}
\mathrm{Re}(\Omega^{0,q})\oplus \Gamma(S_+^{\mathrm{can}})\\ if&\longmapsto
\left(
-2i\,df,\,
0,\,
i\lambda f\psi_0
\right) 
\end{align*}
We complement the linearisation of the curvature equation with the appropriate gauge-fixing equation, i.e., the $L^2$ adjoint of the infinitesimal gauge action at the canonical solution (the calculation required to determine the $L^2$ adjoint is available at \S\S 2.2.2. in \cite{Nicol}):
\begin{equation}
-2id^*a-
i\lambda\mathrm{Im}\langle \psi_0,\sigma\rangle=0
\label{eq:gauge}
\end{equation}
Similar to part 1, we would like to write the equations \eqref{linc} and \eqref{eq:gauge} together in a specific form. We need some preparation to do that. Define 
\begin{align*}
{\mathcal W}
=
i\mathrm{Re}(\Omega^0)
\oplus
i\mathrm{Re}(\Omega^0)\,\omega
\oplus
\bigoplus_{\substack{2\leq p\leq m\\p\ {\rm even}}}
s_p\mathrm{Re}(\Omega^{0,p})
\end{align*}
The first summand is the target of the gauge fixing equation \eqref{eq:gauge}. Define a map
\begin{align*}
\Phi_0:
i\mathrm{Re}(\Omega^0)
\oplus
i\mathrm{Re}(\Omega^0)\,\omega
\longrightarrow
\Omega^{0,0},\,\,\,
\Phi_0(if,ih\omega)
=
\frac12f-\frac{im}{2}h\label{eq:Phi0}
\end{align*}
For every even $p>0$, an element of
$s_p\mathrm{Re}(\Omega^{0,p})$ can be written uniquely in the form
\begin{equation}
\theta_p
=
\frac{s_p}{(\sqrt2)^p}
(\tau_p+\overline{\tau_p}),
\qquad
\tau_p\in\Omega^{0,p}
\label{eq:theta-coordinate}
\end{equation}
Put
\begin{equation}
t_p=\frac{r_{p-1}}2
\label{eq:tp-def}
\end{equation}
and define
\begin{equation}
\Phi_p(\theta_p)=t_p\tau_p
\label{eq:Phip-def}
\end{equation}
Together, these maps define a real bundle isomorphism
\begin{equation}
\Phi:
{\mathcal W}
\longrightarrow S_+^{\mathrm{can}}
\label{eq:Phi-def}
\end{equation}
Define
\begin{equation}
P(a,b)
=
\left(
2id^*a,\,
\Pi(2i\,da+F_b)
\right)
\label{eq:P-def}
\end{equation}
Notice $P$ is one part of the linearised curvature equation with appropriate gauge fixing. We claim that
\begin{equation}
\Phi P(a,b)
=
D_{A^{\mathrm{Ch}}}\eta+R\eta
\label{eq:P-dirac}
\end{equation}
for a zeroth-order real bundle map
\begin{align*}
R:S_-^{\mathrm{can}}\longrightarrow S_+^{\mathrm{can}}
\end{align*}
Next we verify this explicitly. Fix an odd $q\geq3$. Recall from 
\eqref{eq:def-alphaq},
$
b_q
=
2^{-\frac{q}{2}}
(\alpha_q+\overline{\alpha_q}).
$ The $(0,q+1)$ principal part of $2s_{q+1}db_q$ is $
2s_{q+1}2^{-\frac{q}{2}}\bar\partial\alpha_q.$ We calculate
\begin{align}
\Phi_{q+1}(2s_{q+1}db_q)
&=
\frac{2\sqrt2}{(\sqrt 2)^{q+1}}\,s_{q+1}\Phi_{q+1}(\bar\partial\alpha_q)
\nonumber\\
&=
\sqrt2\,r_q\bar\partial\alpha_q
\nonumber\\
&=
\sqrt2\,\bar\partial\eta_q\qquad[\text{definition \eqref{eq:eta-def} is used here}]
\label{eq:d-up}
\end{align}
Similarly, one calculates
\begin{align}
\Phi_{q-1}
(2\delta_qs_{q-1}d^*b_q)
&=
\frac{\sqrt2}{(\sqrt{2})^{q-1}}\,\delta_q s_{q-1}
\Phi_{q-1}(\bar\partial^*\alpha_q)
\nonumber\\
&=
\frac{\sqrt2}{2}\,\delta_q r_{q-2}\bar\partial^*\alpha_q
\nonumber\\
&=
\sqrt2\,r_q\bar\partial^*\alpha_q\qquad[\text{definition \eqref{eq:r-recursion} is used here}]\nonumber\\
&=
\sqrt2\,\bar\partial^*\eta_q\qquad[\text{definition \eqref{eq:eta-def} is used here}]
\label{eq:d-down}
\end{align}
Finally we calculate the $q=1$ case. 
\begin{comment}
write
\begin{align*}
u=\bar\partial^*\alpha_1.
\end{align*}
At the level of principal symbols,
\begin{equation}
d^*a=\sqrt2\,\mathrm{Re} u,
\qquad
\Lambda da=-\sqrt2\,\Im u.
\label{eq:scalar-symbols}
\end{equation}
The projection of $2i\,da$ onto $i\mathbb R\omega$ is
\begin{align*}
\frac{2i}{m}\Lambda(da)\,\omega.
\end{align*}
Hence
\end{comment}
\begin{align}
\Phi_0
\left(
2id^*a,\,
\Pi_{i\mathrm{Re}(\Omega^0)\omega}(2i\,da)
\right)
&=
\Phi_0
\left(
2id^*a,\,\frac{2i}{m}\Lambda(da)\,\omega
\right)
\nonumber\\
&=d^*a-i\Lambda(da)\nonumber\\
&=2\bar\partial^*a\nonumber\\
&=\sqrt{2}\bar\partial^*\alpha_1
\label{eq:scalar-Dminus}
\end{align}
Likewise, the $(0,2)$ principal part of $2i\,da$, under
$\Phi_2$, is
\begin{equation}
\sqrt2\,\bar\partial\alpha_1
\label{eq:a-up}
\end{equation}
Combining \eqref{eq:d-up}, \eqref{eq:d-down},
\eqref{eq:scalar-Dminus}, and \eqref{eq:a-up}, the complete
first-order part of $\Phi P$ is
\begin{align*}
\sqrt2(\bar\partial+\bar\partial^*)\eta
=
D_{A^{\mathrm{Ch}}}\eta
\end{align*}
On an almost complex manifold,
\begin{align*}
d=\partial+\bar\partial+N+\bar N
\end{align*}
where $N$ and $\bar N$ are zeroth-order operators determined by the
Nijenhuis tensor.  Similarly,
\begin{align*}
d^*
=
\partial^*+\bar\partial^*+N^*+\bar N^*
\end{align*}
Thus all terms omitted in the above principal-part computation are
zeroth order.  This proves the claim \eqref{eq:P-dirac}.\par
Now we move on to the other part of the linearisation of the curvature equation. Define
\begin{equation}
Q(\sigma)
=
\left(
-i\mathrm{Im}\langle\psi_0,\sigma\rangle,\,
\Pi(d_{\psi_0}q(\sigma))
\right)
\label{eq:Q-def}
\end{equation}
Equations \eqref{linc} and \eqref{eq:gauge} are precisely
\begin{equation}
P(a,b)=\lambda Q(\sigma)
\label{eq:P=lambdaQ}
\end{equation}
Set
\begin{equation}
H=\Phi Q:S_+^{\mathrm{can}}\longrightarrow S_+^{\mathrm{can}}
\label{eq:H-def}
\end{equation}
Next we determine $H$ on every degree. 
\begin{comment}
Recall that, if
\begin{align*}
N=2^{m-1}=\operatorname{rk}_{\mathbb C}S^+,
\end{align*}
then
\begin{align*}
E_\phi
=
\phi\otimes\phi^*
-
\frac{|\phi|^2}{N}\mathrm{Id}.
\end{align*}
Therefore
\begin{equation}
(dE)_{\psi_0}(\sigma)
=
\sigma\otimes\psi_0^*
+
\psi_0\otimes\sigma^*
-
\frac{2\mathrm{Re}\langle\sigma,\psi_0\rangle}{N}\mathrm{Id}.
\label{eq:dE}
\end{equation}
\end{comment}
Let $p>0$ be even and $
\sigma_p\in\Omega^{0,p}$. Define a real bundle isomorphism
\begin{equation}
j_p:
s_p\mathrm{Re}(\Omega^{0,p})
\longrightarrow
\Omega^{0,p},\qquad
j_p(\theta)=\mathrm{cl}(\theta)\psi_0
\label{eq:jp}
\end{equation}
\begin{comment}
If
\begin{align*}
\theta
=
\frac{s_p}{(\sqrt2)^p}
(\tau+\bar\tau),
\end{align*}
then
\begin{equation}
j_p(\theta)=s_p\tau.
\label{eq:jp-coordinate}
\end{equation}
Thus $j_p$ is a real bundle isomorphism.
\end{comment}
We now show that
\begin{equation}
j_p
\left(
\Pi(d_{\psi_0}q(\sigma_p))
\right)
=
h_p\sigma_p\qquad\text{for a constant $h_p>0$}
\label{eq:hp}
\end{equation}
The difficulty is that $q$ is not defined by an explicit formula for its form components. It is defined implicitly through Clifford multiplication: $\mathrm{cl}(q(\phi))=E_\phi$. Recall that $
E_\phi
=
\phi\otimes\phi^*
-
\frac{1}{2^{m-1}}|\phi|^2\mathrm{Id}$.
Therefore
\begin{equation}
d_{\psi_0}E(\sigma)
=
\sigma\otimes\psi_0^*
+
\psi_0\otimes\sigma^*
-
\frac{2\mathrm{Re}\langle\sigma,\psi_0\rangle}{2^{m-1}}\mathrm{Id}
\label{eq:dE}
\end{equation}
Let $\theta\in s_p\mathrm{Re}(\Omega^{0,p})$. Since $\mathrm{cl}(\theta)$ is trace free, equation \eqref{eq:dE} gives us
\begin{equation}
\left\langle
d_{\psi_0}E(\sigma_p),\mathrm{cl}(\theta)
\right\rangle
=
2\mathrm{Re}
\langle
\mathrm{cl}(\theta)\psi_0,\sigma_p
\rangle
\label{eq:dE-pairing}
\end{equation}
The inner product on the left hand side is on End$(S_+^{\mathrm{can}})$ and the one on the right hand side is on $S_+^{\mathrm{can}}$. For the degree-$p$ summand in $i\mathfrak{su}(S_+^{\mathrm{can}})$ (see \eqref{4m-Clifford-isomorphism}, \eqref{4m-2-Clifford-isomorphism}), there is a positive constant $\kappa_{m,p}$ such that
\begin{equation}
\langle \mathrm{cl}(\theta_1),\mathrm{cl}(\theta_2)\rangle
=
\kappa_{m,p}
\langle\theta_1,\theta_2\rangle
\label{eq:clifford-metric}
\end{equation}
Likewise, since $j_p$ is a $U(m)$-equivariant isomorphism, there is
a positive constant $b_{m,p}$ such that
\begin{equation}
\mathrm{Re}\langle
j_p(\theta_1),j_p(\theta_2)
\rangle
=
b_{m,p}
\mathrm{Re}\langle\theta_1,\theta_2\rangle
\label{eq:jp-metric}
\end{equation}
Since $\mathrm{cl}(q(\phi))=E_\phi$, differentiating at $\psi_0$, we get $$\mathrm{cl}\big(d_{\psi_0}q(\sigma_p)\big)=d_{\psi_0}E(\sigma_p)$$
Moreover, because $\theta\in s_p\mathrm{Re}(\Omega^{0,p})$, \eqref{eq:clifford-metric} tells us $$\left\langle\mathrm{cl}\big(d_{\psi_0}q(\sigma_p)\big),\mathrm{cl}(\theta)
\right\rangle=\left\langle\mathrm{cl}\big(\Pi(d_{\psi_0}q(\sigma_p))\big),\mathrm{cl}(\theta)\right\rangle$$
Now take an arbitrary $\tau\in\Omega^{0,p}$ and put $\theta=j_p^{-1}(\tau)$ in equation \eqref{eq:dE-pairing}. We get
\begin{align*}
    \left\langle \mathrm{cl}\big(\Pi(d_{\psi_0}q(\sigma_p))\big),\mathrm{cl}(j_p^{-1}(\tau))\right\rangle=2\mathrm{Re}\left\langle \tau,\sigma_p\right\rangle
\end{align*}
Identity \eqref{eq:clifford-metric} tells us that this is equivalent to
\begin{align*}
    \kappa_{m,p}\left\langle \Pi(d_{\psi_0}q(\sigma_p)),j_p^{-1}(\tau)\right\rangle=2\mathrm{Re}\left\langle \tau,\sigma_p\right\rangle
\end{align*}
Identity \eqref{eq:jp-metric} tells us that this is again equivalent to
\begin{align}
    \frac{\kappa_{m,p}}{b_{m,p}}\mathrm{Re}\left\langle j_p\big(\Pi(d_{\psi_0}q(\sigma_p))\big),\tau\right\rangle=2\mathrm{Re}\left\langle\sigma_p,\tau\right\rangle
\end{align}
Since $\tau$ was arbitrarily chosen, this proves \eqref{eq:hp}:
\begin{align*}
    j_p
\left(
\Pi(d_{\psi_0}q(\sigma_p))
\right)
=
h_p\sigma_p\qquad\text{for a constant $h_p=\frac{2b_{m,p}}{\kappa_{m,p}}>0$}
\end{align*}
Now if $
\Pi((dq)_{\psi_0}(\sigma_p))
=
\frac{s_p}{(\sqrt2)^p}
(\tau_p+\bar\tau_p)$ for some $\tau_p\in\Omega^{0,p}$, we have 
\begin{align*}
    \Phi_p\big(\Pi((dq)_{\psi_0}(\sigma_p))\big)&=t_p\tau_p\\
    &=\frac{t_p}{s_p} j_p\big(\Pi((dq)_{\psi_0}(\sigma_p))\big)\\
    &=\frac{t_p h_p}{s_p}\sigma_p
\end{align*}
Therefore
\begin{equation}
H|_{\Omega^{0,p}}
=
H_p\mathrm{Id},
\qquad
H_p=\frac{t_p h_p}{s_p}
\label{eq:Hp}
\end{equation}
We now treat the $\Omega^{0,0}$ component. Write
\begin{align*}
\sigma_0=(x+iy)\psi_0,
\qquad x,y\in\mathrm{Re}(\Omega^0)
\end{align*}
We have 
\begin{align*}
    Q(\sigma_0)=\Big(-i\mathrm{Im}\langle\psi_0,(x+iy)\psi_0\rangle, \Pi\big(d_{\psi_0}q((x+iy)\psi_0)\big)\Big)
\end{align*}
We get $-i\mathrm{Im}\langle\psi_0,(x+iy)\psi_0\rangle=iy$ and since $q(e^{it}\psi_0)=q(\psi_0)$, we also have $d_{\psi_0}q(i\psi_0)=0$. Moreover $d_{\psi_0}q(\psi_0)=2q(\psi_0)$. Therefore $ Q(\sigma_0)$ reads
\begin{align*}
    Q(\sigma_0)=\big(iy,2x\Pi(q(\psi_0))\big)
\end{align*}
In the proof of theorem \ref{vortex} we show that 
\begin{align*}
    \Pi(q(\psi_0))=\frac{ic_m}{m}\omega,\qquad \text{for a constant $c_m>0$}
\end{align*}
Therefore we get 
\begin{align}
    Q(\sigma_0)=\big(iy,\frac{2ixc_m}{m}\omega\big)\qquad \text{and\,\,\, } H(\sigma_0)=\frac{y}{2}-ic_mx
\end{align}
In the real basis $(1,i)$,
\begin{align*}
H_0:=H|_{\Omega^0}
=
\begin{bmatrix}
0&\frac12\\
-c_m&0
\end{bmatrix}
\end{align*}
Therefore the linearisation of the curvature equation along with the appropriate gauge fixing condition reads
\begin{align}
    (D_{A^{\mathrm{Ch}}}+R)\eta=\lambda H\sigma
\end{align}
with $H$ described as above and this concludes part 2.\par
\medskip
\noindent\textit{Part 3.} We start with a lemma which will be essential in proving the strong ellipticity of $\mathcal A$ in this part. 
\begin{lemma}\label{positivity}
For every even integer $p>0$ and every odd integer $q$ such that
$
|p-q|=1,
$
the degreewise coefficients $H_p$ and $G_q$ satisfy
\[
\mathrm{Re}\left(H_p^{-1}G_q^{-1}\right)>0
\]
Moreover, since there are only finitely many adjacent degree pairs, there is a constant $\kappa_+>0$, depending only on the dimension, such that
\[
\mathrm{Re}\left(H_p^{-1}G_q^{-1}\right)
\geq \kappa_+
\]
for every such pair $(p,q)$.
\end{lemma}
\begin{proof}
Let $p>0$ be even. First suppose $
p=q+1$, $3\leq q=2k+1<m$. Since $
t_p=\frac{r_q}{2}$, equations \eqref{eq:Gq-def} and \eqref{eq:Hp} give us
\begin{equation}
G_qH_p
=
\frac{h_p}{2}
\frac{A_q}{s_{q+1}}
\label{eq:HG-up}
\end{equation}
Since $q$ is odd $s_q=s_{q+1}$. Therefore we get
\begin{align*}
    \frac{A_q}{s_{q+1}}=\frac{A_q}{s_q}=1+\frac{(-1)^{m+k+1}i^ms_{2m-q}}{s_q}
\end{align*}
Moreover using the identity: $s_j^2=(-1)^{\frac{j(j+1)}{2}}$, we notice that 
\begin{align}\label{id 1}
    \bigg(\frac{(-1)^{m+k+1}i^ms_{2m-q}}{s_q}\bigg)^2=(-1)^{m}(-1)^{m-q}=-1
\end{align}
Since $h_p>0,$ this tells us that Re$(G_qH_p)>0.$\par 
Next we deal with the case $p=q-1,3\leq q=2k+1<m$. Similarly we have 
\begin{align*}
    G_qH_p=\frac{A_q t_p h_p}{r_qs_{q-1}}=\frac{A_q r_{p-1}h_p}{2r_qs_{q-1}}=\frac{A_q h_p}{s_{q-1}\delta_q}=\frac{A_q h_p\delta_q}{s_{q-1}}
\end{align*}
Using the definition of $A_q$ again we get
\begin{align*}
    \frac{A_q\delta_q}{s_{q-1}}&=\frac{\delta_q s_q}{s_{q-1}}\bigg(1+\frac{(-1)^{m+k+1}i^ms_{2m-q}}{s_q}\bigg)\\
    &=\delta_q (-1)^k i\bigg(1+\frac{(-1)^{m+k+1}i^ms_{2m-q}}{s_q}\bigg)\qquad[\text{since for $q$ odd, } \frac{s_q}{s_{q-1}}=(-1)^ki]
\end{align*}
Now let's analyse both the terms $\delta_q (-1)^ki$ and $\frac{(-1)^{m+k+1}i^ms_{2m-q}}{s_q}$ carefully. We start with $m=2l$ even. We get
\begin{align*}
    \delta_q(-1)^k i=(-1)^{l+k+1}(-1)^ki=(-1)^{l+1}i\\
    \frac{(-1)^{m+k+1}i^ms_{2m-q}}{s_q}=(-1)^{l+k+1}\frac{s_{(-2k-1)}}{s_{2k+1}}=(-1)^{l+k+1}\frac{s_{2k-1}}{s_{2k+1}}=(-1)^{l+k+1}\frac{s_{2k}}{s_{2k+1}}=(-1)^li
\end{align*}
We have used here that $s_k$ is $4$-periodic and $\frac{s_q}{s_{q-1}}=(-1)^ki$ for $q=2k+1$ odd. Similarly for $m=2l+1$ odd we get
\begin{align*}
    \delta_q(-1)^k i=(-1)^{\frac{m+1}{2}}(-1)^ki=(-1)^{l+k+1}i\\
     \frac{(-1)^{m+k+1}i^ms_{2m-q}}{s_q}=(-1)^{l+k}i\frac{s_{4l+2-2k-1}}{s_{2k+1}}=(-1)^{l+k}i
\end{align*}
The calculations above tell us that 
\begin{align*}
    \frac{(-1)^{m+k+1}i^ms_{2m-q}}{s_q}=-1\times(\delta_q(-1)^k i)
\end{align*}
Along with identity \eqref{id 1} this tells us
\begin{align*}
    \mathrm{Re}\big(\frac{A_q\delta_q}{s_{q-1}}\big)=1,\qquad \text{therefore }\,\,\mathrm{Re}(G_qH_p)=h_p>0
\end{align*}
There are two cases left: when $q=1$ and for odd $m,$ when $q=m$. For $q=1$ we get
\begin{align*}
\frac{A_1}{s_2}
=
\frac{i}{i}
=
1
\end{align*}
When $m$ is odd and $q=m$, write $m=2\ell+1$.  It follows that 
\begin{align*}
    \delta_m\frac{A_m}{s_{m-1}}&=(-1)^{l+1}(-1)^{m+\frac{m(m+1)}{2}}i^m\frac{s_m}{s_{m-1}}\\
    &=(-1)^{l+1}(-1)^{2l+1+(2l+1)(l+1)}(-1)^li(-1)^l i\\
    &=1
\end{align*}
Thus for every adjacent even-odd pair, $
\mathrm{Re}(G_qH_p)>0$. Since these are non-zero complex numbers,
\begin{align*}
\mathrm{Re}\bigl((G_qH_p)^{-1}\bigr)
=
\frac{\mathrm{Re}(G_qH_p)}{|G_qH_p|^2}
>0
\end{align*}
There are only finitely many degrees, so there exists
$
\kappa_+>0
$
such that
\begin{equation}
\mathrm{Re}\bigl((G_qH_p)^{-1}\bigr)
\geq\kappa_+
\label{eq:inverse-bound}
\end{equation}
for every adjacent degree pair.
\end{proof}
The term $
H^{-1}RG^{-1}D_{A^{\mathrm{Ch}}}$ has order one, because $H^{-1},R$ and $G^{-1}$ are all zeroth order bundle maps.  It therefore makes no contribution to the principal symbol of $\mathcal A$ which is a second order operator. Hence
\begin{equation}
\mathrm{Sym}(\mathcal A)(v^*)
=
H^{-1}\mathrm{Sym}(D_{A^{\mathrm{Ch}}})(v^*)\,
G^{-1}\mathrm{Sym}(D_{A^{\mathrm{Ch}}})(v^*)
\label{eq:A-principal-symbol}
\end{equation}
We have 
\begin{align*}
 \mathrm{Sym}(D_{A^{\mathrm{Ch}}})(v^*)=i\mathrm{cl}(v^*)  
\end{align*}
It remains to understand the insertion of $G^{-1}$ and $H^{-1}$ between the two Dirac symbols.\par 
Fix a point $x\in M$, an even integer $p>0$ and a non-zero covector $v^*\in T^*_xM$. Pick a spinor $\zeta\in\Lambda^{0,p}_x$. Choose a unitary coframe $
\theta^1,\ldots,\theta^m$ at $x$ such that $(v^*)^{0,1}$ is a nonzero multiple of $\bar\theta^1$.
The standard basis of $\Lambda^{0,p}_x$ is formed by the elements
\[
\bar\theta^I
=
\bar\theta^{i_1}\wedge\cdots\wedge\bar\theta^{i_p},
\qquad
I=\{i_1<\cdots<i_p\}
\]
Clifford multiplication by a covector is given by the formula \eqref{Clifford one form}. Consequently,
for a standard basis vector $\bar\theta^I$, exactly one of the following two possibilities occurs. If $1\notin I$, the contraction term vanishes and exterior multiplication by $\bar\theta^1$ is nonzero.  Thus $ \mathrm{cl}(v^*)\bar\theta^I\in\Lambda^{0,p+1}_x$. If $1\in I$, exterior multiplication by $\bar\theta^1$ vanishes and the contraction
term is nonzero. Thus $\mathrm{cl}(v^*)\bar\theta^I\in\Lambda^{0,p-1}_x$. Therefore the symbol of $D_{A^{\mathrm{Ch}}}$ carries every standard basis vector in
$\Lambda^{0,p}_x$ into exactly one adjacent odd degree $
q=p+1$ or
$q=p-1$.\par
Let $e_I=\bar\theta^I$ be one such basis vector, and suppose that
$
\mathrm{Sym}(D_{A^{\mathrm{Ch}}})(v^*)e_I\in\Lambda^{0,q}_x.
$ Since $G^{-1}$ acts on $\Lambda^{0,q}$ by the complex scalar $G_q^{-1}$, we have
\[
G^{-1}\mathrm{Sym}(D_{A^{\mathrm{Ch}}})(v^*)e_I
=
G_q^{-1}\mathrm{Sym}(D_{A^{\mathrm{Ch}}})(v^*)e_I
\]
The Dirac symbol is complex-linear, so the scalar $G_q^{-1}$ can be pulled through the second Dirac symbol:
\begin{align*}
\mathrm{Sym}(D_{A^{\mathrm{Ch}}})(v^*)G^{-1}
\mathrm{Sym}(D_{A^{\mathrm{Ch}}})(v^*)e_I
&=G_q^{-1}
\mathrm{Sym}(D_{A^{\mathrm{Ch}}})(v^*)
\mathrm{Sym}(D_{A^{\mathrm{Ch}}})(v^*)e_I\\
&=G_q^{-1}|v^*|^2e_I
\end{align*}
Since $H^{-1}$ acts on $\Lambda^{0,p}$ by $H_p^{-1}$, \eqref{eq:A-principal-symbol} gives
\begin{equation}
\mathrm{Sym}(\mathcal A)(v^*)e_I
=
|v^*|^2H_p^{-1}G_q^{-1}e_I
=
|v^*|^2(G_qH_p)^{-1}e_I
\label{eq:symbol-on-basis}
\end{equation}
Thus the coefficient that occurs in the second-order symbol is precisely
$(G_qH_p)^{-1}$.  By lemma \ref{positivity},
\[
\mathrm{Re}\big((G_qH_p)^{-1}\big)\geq\kappa_+>0
\]
Hence
\begin{equation}
\mathrm{Re}\left\langle
\mathrm{Sym}(\mathcal A)(v^*)e_I,e_I
\right\rangle
\geq
\kappa_+|v^*|^2|e_I|^2
\label{eq:basis-positive}
\end{equation}
We now pass from the standard basis vectors to an arbitrary element of
$\Lambda^{0,p}_x$.  Define
\[
V_p^\uparrow(v^*)
=
\operatorname{span}\{\bar\theta^I:1\notin I\},
\qquad
V_p^\downarrow(v^*)
=
\operatorname{span}\{\bar\theta^I:1\in I\}
\]
Then
\[
\Lambda^{0,p}_x
=
V_p^\uparrow(v^*)\oplus V_p^\downarrow(v^*)
\]
is an orthogonal decomposition.  The symbol of $D_{A^{\mathrm{Ch}}}$ maps
$V_p^\uparrow(v^*)$ into $\Lambda^{0,p+1}$ and
$V_p^\downarrow(v^*)$ into $\Lambda^{0,p-1}$. Therefore
\eqref{eq:symbol-on-basis} shows that
\begin{align*}
\mathrm{Sym}(\mathcal{A})(v^*)
=
|v^*|^2(G_{p+1}H_p )^{-1}\mathrm{Id}
\quad\text{on }V_p^\uparrow(v^*)
\end{align*}
and
\begin{align*}
 \mathrm{Sym}(\mathcal A)(v^*)
=
|v^*|^2(G_{p-1}H_p )^{-1}\mathrm{Id}
\quad\text{on }V_p^\downarrow(v^*)   
\end{align*}
Write
\[
\zeta=\zeta^\uparrow+\zeta^\downarrow
\]
according to this orthogonal decomposition.  Since the two summands are orthogonal
and are preserved by $\mathrm{Sym}(\mathcal A)(v^*)$, we obtain
\begin{align*}
\mathrm{Re}\left\langle
\mathrm{Sym}(\mathcal A)(v^*)\zeta,\zeta
\right\rangle
&=
|v^*|^2
\left(\mathrm{Re}(G_{p+1}H_p )^{-1}\right)
|\zeta^\uparrow|^2+
|v^*|^2
\left(\mathrm{Re}(G_{p-1}H_p )^{-1}\right)
|\zeta^\downarrow|^2\\
&\geq
\kappa_+|v^*|^2
\left(
|\zeta^\uparrow|^2+|\zeta^\downarrow|^2
\right)\\
&=
\kappa_+|v^*|^2|\zeta|^2
\end{align*}
for every $\zeta\in\Lambda^{0,p}_x$ with $p>0$ even.\par
It remains to treat $\Lambda^{0,0}_x$.  This case is different because $H_0$ is real-linear rather than multiplication by a complex scalar. Let
$
\zeta\in\Lambda^{0,0}_x.
$
The symbol $\mathrm{Sym}(D_{A^{\mathrm{Ch}}})(v^*)$ sends $\Lambda^{0,0}_x$ into
$\Lambda^{0,1}_x$.  Hence $G^{-1}$ acts on the intermediate spinor by $G_1^{-1}$.
Since
\[
G_1^{-1}=\Big(\frac{A_1}{r_1}\Big)^{-1}=-i
\]
is a complex scalar and the Dirac symbol is complex-linear,
\[
\begin{aligned}
\mathrm{Sym}(\mathcal A)(v^*)\zeta
&=
H_0^{-1}
\mathrm{Sym}(D_{A^{\mathrm{Ch}}})(v^*)
G_1^{-1}
\mathrm{Sym}(D_{A^{\mathrm{Ch}}})(v^*)\zeta\\
&=
H_0^{-1}(-i)
\mathrm{Sym}(D_{A^{\mathrm{Ch}}})(v^*)
\mathrm{Sym}(D_{A^{\mathrm{Ch}}})(v^*)\zeta\\
&=
|v^*|^2H_0^{-1}(-i)\zeta
\end{aligned}
\]
Write
\[
\zeta=(x+iy)\psi_0\in\Lambda^{0,0}_x,\qquad{then }\,\,\,(-i\zeta)=(y-ix)\psi_0
\]
We get
\begin{align*}
    H_0^{-1}(-i\zeta)=\begin{bmatrix}
    0&-\frac{1}{c_m}\\2&0
    \end{bmatrix}\begin{bmatrix}
    y\\-x
    \end{bmatrix}=\begin{bmatrix}
    \frac{x}{c_m}\\2y
    \end{bmatrix}=\Big(\frac{x}{c_m}+2iy\Big)\psi_0
\end{align*}
Therefore we have
\[
\begin{aligned}
\left\langle
\text{Sym}(\mathcal A)(v^*)\zeta,\zeta
\right\rangle
&=
|v^*|^2
\left(c_m^{-1}x^2+2y^2\right)\\
&\geq
\kappa_0|v^*|^2(x^2+y^2)\\
&=
\kappa_0|v^*|^2|\zeta|^2
\end{aligned}
\]
where
\[
\kappa_0:=\text{min}\{c_m^{-1},2\}>0
\]
The even-degree decomposition
\[
(S_+^{\mathrm{can}})_x
=
\bigoplus_{\substack{0\leq p\leq m\\ p\ \mathrm{even}}}
\Lambda^{0,p}_x
\]
is orthogonal.  The preceding calculation also shows that
$\mathrm{Sym}(\mathcal A)(v^*)$ preserves each even-degree summand.  Thus, if
\[
\zeta=\sum_{\substack{0\leq p\leq m\\ p\ \mathrm{even}}}\zeta_p,
\qquad
\zeta_p\in\Lambda^{0,p}_x
\]
then
\[
\begin{aligned}
\mathrm{Re}\left\langle
\mathrm{Sym}(\mathcal A)(v^*)\zeta,\zeta
\right\rangle
&=
\mathrm{Re}\sum_{\substack{0\leq p\leq m\\ p\ \mathrm{even}}}
\left\langle
\mathrm{Sym}(\mathcal A)(v^*)\zeta_p,\zeta_p
\right\rangle\\
&\geq
\kappa|v^*|^2
\sum_{\substack{0\leq p\leq m\\ p\ \mathrm{even}}}|\zeta_p|^2\\
&=
\kappa|v^*|^2|\zeta|^2
\end{aligned}
\]
where
\[
\kappa:=\min\{\kappa_0,\kappa_+\}>0\qedhere
\]
\end{proof}
\subsubsection*{Solution on K\"ahler manifolds: vortices}
\begin{proof}[Proof of theorem \ref{vortex}]
    Gauduchon \cite{Paul} tells us that on a K\"ahler manifold, the Dirac equation reads: $D_A=\sqrt{2}(\bar\partial_B+\bar\partial_B^*)$. Since $\varphi\in\Omega^0(M,\mathcal L),$ the Dirac equation \eqref{Dirac} reads
    \begin{align*}
        \bar\partial_B\varphi=0
    \end{align*}
    Notice that since $\varphi\in\Omega^0(M,\mathcal L)$, $E_\varphi$ is $U(m)$ invariant. Therefore its form representative $q(\varphi)$ must lie in the $U(m)$-invariant subspace of even forms. That invariant subspace is generated by the powers of the K\"ahler form. Hence there exists a constant $c_m$ depending only on the dimension of $2m=n$ such that 
    \begin{align*}
        \Pi\big(q(\varphi)\big)=\frac{ic_m}{m}|\varphi|^2\omega
    \end{align*}
    The same reasoning gives us
    \begin{align*}
        \Pi\big(q(\psi_0)\big)=\frac{ic_m}{m}\omega
    \end{align*}
    Therefore the curvature equation \eqref{Curvature} turns into the following.
    \begin{align*}
        \Pi\big(2F_B+\frac{irc_m}{m}\omega\big)=\frac{ic_m}{m}|\varphi|^2\omega
    \end{align*}
    Writing the iRe$(\Omega^0)\omega$ component and the $i$Re$\Omega^{0,2}$ component of the curvature equation separately we get
    \begin{align*}
        i\Lambda F_B=\frac{c_m}{2}(r-|\varphi|^2)\\
        F_B^{0,2}=0
    \end{align*}
    Next we show that for $m\geq 3$, $c_m=\frac{m}{2^{m-1}}$. Recall $
\mathrm{cl}(q(\psi_0))=E_{\psi_0}=\psi_0\otimes\psi_0^*-\frac{1}{2^{m-1}}\mathrm{Id}$. We pair it with $\mathrm{cl}(i\omega)$ to get
\[
\begin{aligned}
\left\langle E_{\psi_0},\mathrm{cl}(i\omega)\right\rangle
&=
\left\langle
\psi_0\otimes\psi_0^*,
\mathrm{cl}(i\omega)
\right\rangle \qquad[\text{since $\mathrm{cl}(i\omega)$ is trace-free}]\\
&=
\left\langle
\mathrm{cl}(i\omega)\psi_0,\psi_0
\right\rangle\\
&=m
\end{aligned}
\]
For $m\geq3$, Clifford multiplication on two forms satisfies
\[
\left\langle
\mathrm{cl}(\theta_1),\mathrm{cl}(\theta_2)
\right\rangle
=
2^{m-1}\langle\theta_1,\theta_2\rangle
\]
Indeed, the proportionality constant can be computed using a unit simple
two form $\theta=i e^1\wedge e^2$.  Since
\[
\mathrm{cl}(\theta)^*\mathrm{cl}(\theta)=\mathrm{Id}_{S_+}
\]
we obtain
\[
|\mathrm{cl}(\theta)|^2
=
\mathrm{tr}(\mathrm{Id}_{S_+})
=
\mathrm{rank}_{\mathbb C}(S_+)=2^{m-1}
\]
Thus
\[
|\mathrm{cl}(i\omega)|^2
=
2^{m-1}|\omega|^2
=
2^{m-1}m
\]
Thereafter we have 
\begin{align*}
    m=\langle \mathrm{cl}(\Pi(q(\psi_0))),\mathrm{cl}(i\omega)\rangle=2^{m-1}\langle\frac{c_m}{m}i\omega,i\omega\rangle=2^{m-1}c_m
\end{align*} 
leading to $c_m=\frac{m}{2^{m-1}}$. When $m=2$, Clifford
multiplication on $S_+$ does not see all of $\Omega^2$, but only the self-dual part $\Omega^2_+$.
Thus the normalization
\[
|\mathrm{cl}(\theta)|^2
=
2^{m-1}|\theta|^2
\]
for arbitrary two forms does not apply in this case.  In dimension four the correct factor on $i\Omega^2_+$ is $4$.  Consequently one obtains instead
\begin{align*}
q(\psi_0)=\frac{i\omega}{4},
\qquad
c_2=\frac12\tag*{\qedhere}  
\end{align*}
\end{proof}
\subsection{Taubes' limit}\label{limit}
Let's take the equations \eqref{Dirac}, \eqref{Curvature} with $r=1$.
\begin{align*}
     D_{A,\beta}\phi=0\\\
    \Pi(F_A-F_{{A^{\mathrm{Ch}}}}+F_\beta)+\frac{ic_m}{m}\omega=\Pi(q(\phi))
\end{align*}
Take a geodesic ball of radius $1$ around a point $x\in M$ and dilate the metric inside the ball by a constant $r>1$. We get a new metric $g_r:=r g$. Since $g(u,v)=\omega(u,Jv)$, the appropriate symplectic form $\omega_r$ with respect to the metric $g_r$ is $r\omega$. Another way of writing this is the following: the fundamental two form associated with \((g_r,J)\) is \(\omega_r=r\omega\). It is closed and coclosed, hence harmonic and satisfies \(|\omega_r|_{g_r}^2=m\). Now we write down the equations above in the metric $g_r$. They read 
\begin{align}
    D_{A_r,\beta_r}\phi_r=0\label{Dirac r}\\
    \Pi(F_{A_r}-F_{{A^{\mathrm{Ch}}_r}}+F_{\beta_r})+\frac{ic_m}{m}\omega_r=\Pi(q(\phi_r))\label{Curvature r}
\end{align}
Take
\begin{align*}
    \phi_r=\frac{1}{\sqrt r}\phi, \, (\beta_{r})_{2k+1}=(\sqrt r)^{2k}\beta_{2k+1} \text{ and } A_r=A
\end{align*}
The Hodge-star operator of $g_r$ and $g$ are related in the following way. $$*_{g_r}|_{\Omega^k}=(\sqrt{r})^{n-2k}*_g|_{\Omega^k}$$
Thereafter $$d^{*_{g_r}}=\frac{1}{r}d^{*_g} $$
A small calculation shows
\begin{align*}
    \mathrm{cl}_{g_r}\big((\beta_r)_{2k+1}\big)=\mathrm{cl}_{g}\Big(\frac{(\sqrt r)^{2k}}{(\sqrt r)^{2k+1}}\beta_{2k+1}\Big)=\frac{1}{\sqrt r}\mathrm{cl}_g(\beta_{2k+1})\\
    \mathrm{cl}_{g_r}\big(*_{g_r}(\beta_r)_{2k+1}\big)=\mathrm{cl}_{g}\Big(\frac{(\sqrt r)^{2k}\times (\sqrt r)^{n-2(2k+1)}}{(\sqrt r)^{n-2k-1}}*_g\beta_{2k+1}\Big)=\frac{1}{\sqrt r}\mathrm{cl}_g(*_g\beta_{2k+1})
\end{align*}
Hence,
\begin{align*}
    \mathrm{cl}_{g_r}(\tilde\beta_r)=\frac{1}{\sqrt{r}}\mathrm{cl}_g(\tilde\beta)
\end{align*}
Similarly we also get
\begin{align*}
     (D_{A_r})_{g_r}=\frac{1}{\sqrt r}(D_A)_g
\end{align*}
Therefore, solving the Dirac equation \eqref{Dirac r} on $g_r$ is equivalent to solving the Dirac equation
$\frac{1}{r}D_{A,\beta}\phi=0$ on $g$. The curvature equation \eqref{Curvature r} can also be thought of as 
\begin{align*}
    \mathrm{cl}_{g_r}\big(\Pi(F_{A_r}-F_{{A^{\mathrm{Ch}}_r}}+F_{\beta_r})+\frac{ic_m}{m}r\omega\big)=\mathrm{cl}_{g_r}\big(\Pi(q(\phi_r))\big)
\end{align*}
This description is more useful to avoid complications coming from the isomorphisms \eqref{4m-Clifford-isomorphism}, \eqref{4m-2-Clifford-isomorphism} which are used to define the quadratic map $q$. Notice 
\begin{align*}
    \mathrm{cl}_{g_r}(d(\beta_r)_{2k+1})=\mathrm{cl}_g\Big(\frac{(\sqrt r)^{2k}}{(\sqrt r)^{2k+2}}d\beta_{2k+1}\Big)=\frac{1}{r}\mathrm{cl}_g(d\beta_{2k+1})\\
    \mathrm{cl}_{g_r}(d^{*_{g_r}}(\beta_r)_{2k+1})=\mathrm{cl}_g\Big(\frac{(\sqrt r)^{2k}}{(\sqrt r)^{2k}\times (\sqrt r)^{2}}d^{*_g}\beta_{2k+1}\Big)=\frac{1}{r}\mathrm{cl}_g(d^{*_g}\beta_{2k+1})
\end{align*}
Moreover since $q(\phi)$ is quadratic in $\phi$, we have 
\begin{align*}
    \mathrm{cl}_{g_r}\big(\Pi(q(\phi_r))\big)=\frac{1}{r}\mathrm{cl}_g\big(\Pi(q(\phi))\big) 
\end{align*}
Thereafter solving the curvature equation \eqref{Curvature r} on $g_r$ is equivalent to solving the equation $\frac{1}{r}\Pi(F_A-F_{{A^{\mathrm{Ch}}}}+F_\beta+\frac{ic_m}{m}r\omega)=\frac{1}{r}\Pi(q(\phi))$ on $g$. In other words, the \textit{zooming in} procedure parallels taking $r\rightarrow\infty$ in the original metric $g$. \par
So, it's natural to ask if there's a Taubes' limit for the equations \eqref{Dirac r}, \eqref{Curvature r}, i.e., say we have solutions to the equations $(A_r,\beta_r,\phi_r)$ for all large enough $r$. We define part of $\phi_r$ lying inside $\Omega^0(M,\mathcal L)$ to be $\sqrt{r}\phi^0_r$. 
\begin{center}
    \textit{Does $(\phi^0_r)^{-1}(0)$ converge to something as $r\rightarrow\infty$?}
\end{center}
Since $\phi^0_r$ is a section of $\mathcal{L},$ generically the zero set of $\phi^{0}_r$ is a codimension $2$ submanifold with homology class the Poincar\'e dual of $c_1(\mathcal L): PD(c_1(\mathcal L))$.\par 
If we are on a K\"ahler manifold, the solution to the equations coming from theorem \ref{vortex} consists of holomorphic sections of $\mathcal L$. In that case, one can choose the solutions corresponding to a fixed effective divisor, so the zero set of the holomorphic component is independent of \(r\) and therefore $(\phi^0_r)^{-1}(0)$ is a complex codimension $1$ complex submanifold and hence indeed there's a Taubes' limit as $r\rightarrow\infty$. One might be hopeful that the Taubes' limit might exist for symplectic manifolds as well and the limit is a \textit{symplectic submanifold} of real codimension $2$. Moreover they will have homology class $PD(c_1(\mathcal L))$. If this ambitious scheme works, the payoff would be more powerful than that of Donaldson's theorem \cite{Simon}, which produces symplectic hypersurfaces Poincar\'e dual to sufficiently large multiples of an integral symplectic class. If \(\mathcal L\) is chosen with \(c_1(\mathcal L)=[\omega]\), this is \(k\,PD(c_1(\mathcal L))\).
\section{Analysis in dimension $6$}\label{sec4}
\subsection{Spin geometry in dimension $6$} \label{spin geometry 6}
In this section we discuss some facts in spin geometry in dimension $6$. These facts are probably known to the experts. However they are still included due to lack of a standard reference.
\subsubsection*{Spin$(6)\cong SU(4)$:} We start with a description of the universal cover of $SO(6),$ i.e., Spin$(6)$. This is one of the exceptional isomorphisms of Lie groups. Enough to show that $SU(4)$ is a double cover of $SO(6)$ since $\pi_1(SO(6))=\mathbb{Z}/2\mathbb{Z}.$ The double cover is constructed as follows:\par
$SU(4)$ has a natural action on $\mathbb{C}^4,$ it induces an action of $SU(4)$ on $\Lambda^2\mathbb{C}^4$. $\Lambda^2\mathbb{C}^4$ is six dimensional (in $\mathbb{C}$) and has a hermitian inner-product induced by the one on $\mathbb{C}^4.$ This can be defined by saying if $\{e_1,\ldots,e_4\}$ is an orthonormal basis for $\mathbb C^4$ then $\{e_i \wedge e_j\}$ is an orthonormal basis for $\Lambda^2 \mathbb C^4$ or more invariantly,
$$
\langle v_1 \wedge v_2, w_1 \wedge w_2 \rangle = \det\langle v_i, w_j\rangle
$$
Where $\langle v, w\rangle$ is the Hermitian inner product on $\mathbb C^4.$ The action of $SU(4)$ on $\Lambda^2\mathbb{C}^4$ preserves this inner product. Moreover we have a hodge star operator $*$ on $\Lambda^2\mathbb{C}^4$ induced by the inner product such that $*^2=1,$ Hence it splits the space into self-dual and anti self-dual forms: 
$$\Lambda^2\mathbb{C}^4=\Lambda^{2}_+\mathbb{C}^4\oplus \Lambda^{2}_-\mathbb{C}^4
$$
This is a decomposition as real vector spaces and the $SU(4)$-equivariant $*$ presented here is a conjugate linear real structure on $\Lambda^2\mathbb{C}^4$. Notice that the action of $SU(4)$ is preserved under this splitting and the real dimension of $\Lambda^{2}_+\mathbb{C}^4$ is $6$. Since, $SU(4)$ preserves the inner product and has determinant one, we get a map from $SU(4)\rightarrow SO(6),$ one checks that the kernel is $\pm$Id, and therefore, by dimensionality reasons, must be a surjection.\par
The discussion above proves the following. 
\begin{lemma}
For $6$-dimensional spin manifold $M$, 
\begin{align*}
    \Lambda^{2}_+(S_+)\cong TM.
\end{align*}
\end{lemma}
\subsubsection*{Positive and negative spin bundles are dual to each other:} For both groups Spin$(6)\cong SU(4)$, it has two irreducible representations of dimension $4$ and they are dual to each other. Since the representations corresponding to positive and negative spinors have different highest weights:  $(1/2,1/2,1/2)$ and  $(1/2,1/2,-1/2)$, the positive and negative spinors must come from the two different irreducible representations and hence they are dual to each other.\par 
If we start from $SU(4)$, which has a standard $4$-dimensional irreducible representation (say $V$). The wedge product $\Lambda^3 V$ is also $4$-dimensional and irreducible. These two representations are not isomorphic to each other because their highest weights are different, thus $V$ and $\Lambda^3 V$ correspond to two spinor representations of Spin$(6)$. However, $SU(4)$ admits an outer automorphism that exchanges the highest weights of $V$ and  $\Lambda^3 V$. So the positive spinor can be either $V$ or $\Lambda^3 V$, depending on the isomorphism between $SU(4)$ and Spin$(6)$ we choose. Despite this outer automorphism, we can still observe that two spinors are dual to each other since the wedge product gives us a non-degenerate pairing: $V\times \Lambda^3 V\rightarrow\mathbb{C}.$
\subsubsection*{Clifford multiplication:} Clifford multiplication of complexified $0$ and $2$-forms give the isomorphisms:
\begin{align*}
    \text{End}_\mathbb{C}(S_+,S_+)\cong \big(\Lambda^{0}(M)\oplus\Lambda^2(M)\big)\otimes\mathbb{C}\cong \text{End}_\mathbb{C}(S_-,S_-)
\end{align*}
For a hermitian endomorphism of spinors of same chirality, we can split it into a trace-free part and the trace-part. $\Lambda^0$ is the trace part and the imaginary $2$-forms (or equivalently the real $4$-forms) (real dimension $15$) act as trace-free hermitian endomorphisms via Clifford multiplication:
\begin{align}\label{Clifford identity 1}
    \mathrm{cl}:i\Lambda^2(M)\cong i\mathfrak{su}(S_+)\cong i\mathfrak{su}(S_-)
\end{align}
\begin{comment}
    \begin{lemma}
    The following diagrams commute.
\begin{align*}
  \begin{tikzcd}
    \Lambda^4(M)\arrow{r}{i *}\arrow[swap]{d}{c}&i\Lambda^2(M)\arrow{d}{c}&&\Lambda^4(M)\arrow{r}{i *}\arrow[swap]{d}{c}&i\Lambda^2(M)\arrow{d}{c}\\
    i\mathfrak{su}(S_{+})\arrow{r}{-Id}&i\mathfrak{su}(S_{+})&&i\mathfrak{su}(S_{-})\arrow{r}{Id}&i\mathfrak{su}(S_{-})
  \end{tikzcd}
\end{align*} 
\end{lemma}
\begin{proof}
\begin{align*}
    \begin{tikzcd}
    e_{i_1}\wedge e_{i_2}\wedge e_{i_3}\wedge e_{i_4}\arrow[r,mapsto]\arrow[d,mapsto]&ie_{i_5}\wedge e_{i_6}\arrow[d,mapsto]\\
    c(e_{i_1})c(e_{i_2})c(e_{i_3})c(e_{i_4})\arrow[r,mapsto,"?"]{-Id}&ic(e_{i_5})c(e_{i_6})
  \end{tikzcd}
\end{align*} 
On $S_{+},ic({e_1})c({e_2})c({e_3})c({e_4})c({e_5})c({e_6})$ is the identity element. From the definition of the Hodge-$*$ operator,\\
$e_{i_1}\wedge e_{i_2}\wedge e_{i_3}\wedge e_{i_4}\wedge e_{i_5}\wedge e_{i_6}=e_1\wedge e_2\wedge e_3\wedge e_4\wedge e_5\wedge e_6\\
\Rightarrow c(e_{i_1})c(e_{i_2})c(e_{i_3})c(e_{i_4})c(e_{i_5})c(e_{i_6})=c({e_1})c({e_2})c({e_3})c({e_4})c({e_5})c({e_6})=i\cdot \text{Id}$\hspace{4 ex} on $S_{+}$\\
$\Rightarrow c(e_{i_1})c(e_{i_2})c(e_{i_3})c(e_{i_4})c(e_{i_5})c(e_{i_6})c(e_{i_5})c(e_{i_6})=i\cdot c(e_{i_5})c(e_{i_6})$\hspace{5 ex} on $S_{+}$\\
$\Rightarrow c(e_{i_1})c(e_{i_2})c(e_{i_3})c(e_{i_4})=-ic(e_{i_5})c(e_{i_6})$\hspace{5 ex} on $S_{+}$\\
\\
Identical calculation proves the commutativity of the second diagram.
\end{proof}
\end{comment}
On $\Lambda^{3}(M;\mathbb{C}),*^2=-1.$ Hence $\Lambda^3(M;\mathbb{C})$ splits as $\Lambda^{3}_+(X,\mathbb{C})\oplus \Lambda^{3}_-(M;\mathbb{C}).$ Here $*$ is a complex linear extension of the Hodge-star operator on $\Lambda^3(M;\mathbb{R}).$ We use the subscripts $+$ and $-$ respectively will denote the eigen-spaces of $+i$ and $-i$. The Clifford action of $\Lambda^{3}_+(M,\mathbb{C})$ on the negative spinors is trivial and the Clifford action of $\Lambda^3_-(M,\mathbb{C})$ on the positive spinors is trivial. \par Therefore, we have the following isomorphisms:
\begin{align*}
    &\mathrm{cl}:\big(\Lambda^{1}(M)\oplus\Lambda^{3}_+(M)\big)\otimes\mathbb{C}\cong\text{End}_\mathbb{C}(S_+,S_-) \\
    &\mathrm{cl}:\big(\Lambda^{1}(M)\oplus\Lambda^{3}_-(M)\big)\otimes\mathbb{C}\cong\text{End}_\mathbb{C}(S_-,S_+)
\end{align*}
An immediate question arises: \textit{how does one differentiate a one-form from a three form as an element of $\text{End}_\mathbb{C}(S_+,S_-)$ or $\text{End}_\mathbb{C}(S_-,S_+)?$ In other words, what is the induced splitting on the right hand sides from the splitting on the left hand sides of the isomorphisms ?} Notice, $\text{End}_\mathbb{C}(S_-,S_+)\cong S_-^*\otimes S_+\cong S_+\otimes S_+$. This splits into symmetric and anti-symmetric tensors:
$$
    S_+\otimes S_+\cong\text{ Sym}(S_+)\oplus \Lambda^2(S_+)
$$
$\Lambda^1\otimes\mathbb{C}$ identifies with $\Lambda^2(S_+)$ (both have rank $6)$, and $\Lambda^{3}_-\otimes\mathbb{C}$ identifies with $\text{ Sym}(S_+)$ (both have rank $10).$ Similarly, for $\text{End}_\mathbb{C}(S_+,S_-)\cong S_+^*\otimes S_-\cong S_-\otimes S_-\cong\text{ Sym}(S_-)\oplus \Lambda^2(S_-)$; $\Lambda^1\otimes\mathbb{C}$ identifies with $\Lambda^2(S_-)$ (both have rank $6)$, and $\Lambda^{3}_+\otimes\mathbb{C}$ identifies with $\text{ Sym}(S_-)$ (both have rank $10).$\par 
When the Riemannian $6$-manifold $(M,g)$ admits a compatible almost complex structure $J$, one can analyse the identity \eqref{Clifford identity 1} further. The almost complex structure $J$ induces a splitting in both the space of two forms and the spin bundle.
\begin{align*}
    \Lambda^2=\mathrm{Re}(\Lambda^0)\omega\oplus \mathrm{Re}(\Lambda^{0,2})\oplus \mathrm{Re}(\Lambda^{1,1}_{p})\\
    S_+=(\Lambda^{0}\oplus\Lambda^{0,2})\otimes\mathcal L,\,\, S_-=(\Lambda^{0,1}\oplus\Lambda^{0,3})\otimes \mathcal L
\end{align*}
where $\Lambda^{1,1}_{p}$ denotes the primitive $(1,1)$-forms, i.e., $(1,1)$-forms which are pointwise orthogonal to $\omega$. cl$(\omega)$ preserves the splitting of $S_+$ and cl$(\Omega^{0,2})$ sends $\Omega^0(\mathcal L)$ to $\Omega^{0,2}(\mathcal L)$:
\begin{align*}
    \mathrm{cl}:\Lambda^{0,2}\cong \Lambda^0(\mathcal L)\otimes (\Lambda^{0,2}(\mathcal{L}))^*
\end{align*}
The Clifford action of forms in $\Lambda^{1,1}_{p}$ are trivial on $\Omega^0(\mathcal{L})$ and in fact we have
\begin{align}\label{Clifford isom}
    \mathrm{cl}:i\mathrm{Re}(\Lambda^{1,1}_{p})\cong i\mathfrak{su}(\Lambda^{0,2}(\mathcal{L}))
\end{align}
Notice both have rank $8$.
\subsubsection*{Spin$^\mathbb C$ iff almost complex in dimension $6$:}
\begin{lemma}
A $6$-dimensional manifold $M$ is spin$^\mathbb C$ iff it admits an almost complex structure $J$.
\end{lemma}
The first proof of this result implicitly appears in \cite{Charles}, which basically argues that in dimension $6$, the only obstruction of $M$ having an almost complex structure is the second Stiefel--Whitney class $w_2(TM)$ having an integral lift which is also the condition required for $M$ being spin$^\mathbb C$. Here we present a more geometric proof, which the author learnt from Claude LeBrun.
\begin{proof}
    Almost complex $\Rightarrow$ spin$^\mathbb C$ is a very well known fact. The key insight behind the proof of the other direction: spin$^\mathbb C\Rightarrow$ almost complex, is the existence of \textit{pure spinors} in dimension $6$. In dimension $6$, a pure spinor at a point $x\in M$ is a spinor $\psi\in S_{+}{|_x}$, whose annihilator under Clifford multiplication
    \begin{align*}
        \mathrm{Ann}(\psi_x):=\{v\in T_xM\otimes\mathbb C\,|\,\mathrm{cl}(v)\psi=0\}
    \end{align*}
    is a complex $3$ dimensional subspace of $T_xM\otimes \mathbb C$. Using this we can define an almost complex structure $J_x$ such that {Ann}$(\psi_x)$ is exactly the $(1,0)$ tangent space with respect to $J_x$. On the other hand, if we have a compatible almost complex structure $J_x$, one can think of the positive spin bundle as $S_+{|_x}=\oplus_p \Lambda_J^{0,2p}$. Remark \ref{remark Clifford action} tells us that the Clifford action of the $(1,0)$-forms on $\Omega^0$ is trivial. So $J_x$ determines a section of $\mathbb{CP}(S_+|_x)$. Hence we have a well-defined correspondence
    \begin{align*}
        J_x\longleftrightarrow \mathbb{CP}(S_+|_x)
    \end{align*}
    It's a fact that in dimension $6$, all non-zero spinors are pure spinors \cite{Chevalley}. Therefore to define a global almost complex structure $J$, all we need is a nowhere vanishing positive spinor, but it always exists because of dimensional reason since $S_+$ is a complex vector bundle of rank $4>3$.
\end{proof}
\subsection{A priori estimates}\label{estimates}
We prove some a priori estimates for solutions to the $6$-dimensional equations. The same techniques apply to all dimensions, we focus on dimension $6$ to be concrete. In dimension $6$, for $\phi\in\Gamma(S_+),A\in\mathcal A,\beta\in\mathrm{Re}(\Omega^{0,3})$, equations \eqref{Dirac}, \eqref{Curvature} read
\begin{align}
    (D_A+\mathrm{cl}(*\beta))\phi=0\label{Dirac 6d}\\
    \Pi\big(F_A-F_{A^{\mathrm{Ch}}}+2id^*\beta+rq(\psi_0)\big)=\Pi\big(q(\phi)\big)\label{Curvature 6d}
\end{align}
For $S=S^\mathrm{can}\otimes\mathcal{L},$ the associated complex line bundle is $L(S)=K^{-1}\otimes\mathcal{L}^2$. We write a unitary connection $A$ on $L(S)$ as $A=A^{\mathrm{Ch}}+2B$, where $A^\mathrm{Ch}$ is a unitary connection on $K^{-1}$ discussed in \S\ref{symplectic} and $B$ is a unitary connection on $\mathcal{L}$. Take $\phi=\phi^0+\phi^{0,2}\in\Omega^0(\mathcal L)\oplus\Omega^{0,2}(\mathcal L)$ and $\beta=\beta^{0,3}+\overline{\beta^{0,3}},\beta^{0,3}\in\Omega^{0,3}$.\par 
Before proceeding further, let's understand the term $\Pi(q(\phi))$ carefully. The matrix representation of Clifford action of $q(\phi)$ on $S_+=\Omega^0(\mathcal{L})\oplus\Omega^{0,2}(\mathcal{L})$ is:
\begin{align*}
    \begin{bmatrix}
    \frac{3|\phi^0|^2-|\phi^{0,2}|^2}{4}&\overline{\phi^{0,2}}\phi^0 \\\overline{\phi^0}\phi^{0,2}&\big(\langle\,,\phi^{0,2}\rangle-\frac{|\phi^{0,2}|^2}{3}\big)-\frac{3|\phi^0|^2-|\phi^{0,2}|^2}{12}
    \end{bmatrix}
\end{align*}
Notice that the term $\big(\langle\,,\phi^{0,2}\rangle-\frac{|\phi^{0,2}|^2}{3}\big)\in \Gamma\big(i\mathfrak{su}(\Lambda^{0,2}(\mathcal{L})\big)\cong i\mathrm{Re}(\Omega^{1,1}_p)$ (isomorphism \eqref{Clifford isom}). Hence the matrix representation of the Clifford action of $\Pi(q(\phi))$ is:
\begin{align*}
    \begin{bmatrix}
    \frac{3|\phi^0|^2-|\phi^{0,2}|^2}{4}&\overline{\phi^{0,2}}\phi^0 \\\overline{\phi^0}\phi^{0,2}&-\frac{3|\phi^0|^2-|\phi^{0,2}|^2}{12}
    \end{bmatrix}
\end{align*}
The diagonal entries are given by cl$\big(\frac{3|\phi^0|^2-|\phi^{0,2}|^2}{12}i\omega\big)$ (follows from lemma 2.3 in \cite{Partha}) and the off-diagonal entries are given by cl$\big(\frac{1}{2}\overline{\phi^0}\phi^{0,2}\big)$ (follows from lemma \ref{lemma 2}). Moreover we get
\begin{align*}
    &\big\langle \mathrm{cl}\big(\Pi(q(\phi)\big)\phi,\phi\big\rangle\\
    &=\big\langle \frac{3|\phi^0|^2-|\phi^{0,2}|^2}{4}\phi^0+|\phi^0|^2\phi^{0,2}+|\phi^{0,2}|^2\phi^0-\frac{3|\phi^0|^2-|\phi^{0,2}|^2}{12}\phi^{0,2},\phi^0+\phi^{0,2}\big\rangle\\
    &=\frac{1}{12}(9|\phi^0|^4+|\phi^{0,2}|^4+18|\phi^0|^2|\phi^{0,2}|^2)\\
    &\geq \frac{1}{12}(|\phi^0|^4+|\phi^{0,2}|^4+2|\phi^0|^2|\phi^{0,2}|^2)\\
    &=\frac{1}{12}|\phi|^4
\end{align*}
\begin{lemma}
There exists a constant $C>0$ depending only on $r$ and the geometry of $M$ such that if $(\phi,A,\beta)$ solves equations \eqref{Dirac 6d}, \eqref{Curvature 6d} then
\begin{align*}
    \|\phi\|_{C^0}^2\leq C(1+\|(F_A)^{1,1}_p\|_{C^0}+\|\beta\|_{C^0}^2)
\end{align*}
\end{lemma}
\begin{proof}
Say $(\phi,A,\beta)$ solves equations \eqref{Dirac 6d}, \eqref{Curvature 6d}. We start with the identity 
\begin{align*}
    \frac{1}{2}\Delta|\phi|^2=\langle\nabla_{A,\beta}^*\nabla_{A,\beta}\phi,\phi\rangle-|\nabla_{A,\beta}\phi|^2
\end{align*}
Here the Laplacian $\Delta=d^*d$ is the non-negative Laplacian on functions. We complement the identity with the Weitzenbo\"ck formula \eqref{Weitzenbock}, \eqref{Weitzenbock 6}:
\begin{align*}
    D_{A,\beta}^2=\nabla_{A,\beta}^*\nabla_{A,\beta}+\frac{s}{4}+\frac{1}{2}\mathrm{cl}(F_A+2id^*\beta)-2|\beta|^2
\end{align*}
to get
\begin{align}
    &\frac{1}{2}\Delta|\phi|^2\\&\leq-\frac{s}{4}|\phi|^2-\frac{1}{2}\langle\mathrm{cl}(F_A+2id^*\beta)\phi,\phi\rangle+2|\beta|^2|\phi|^2\nonumber\\
    &=\Big(-\frac{s}{4}+2|\beta|^2\Big)|\phi|^2-\frac{1}{2}\langle\mathrm{cl}(\Pi(q(\phi)))\phi,\phi\rangle\\
    &\hspace{0.6 cm}-\frac{1}{2}\langle \mathrm{cl}((\Pi(F_{A^{\mathrm{Ch}}})-rq(\psi_0)
    +(F_A+2id^*\beta)^{1,1}_p)\phi,\phi\rangle\label{eq 1k}\\
    &\leq-\frac{s}{4}|\phi|^2-C|\phi|^4+2|\beta|^2|\phi|^2+C|\phi|^2-\frac{1}{2}\langle\mathrm{cl}((F_A)^{1,1}_p+(2id^*\beta)^{1,1}_p)\phi,\phi\rangle\nonumber
\end{align}
Since $\beta\in$ Re$(\Omega^{0,3}),|(d^*\beta)^{1,1}|=|N\circ \beta|\leq C|\beta|$ where $N$ is the Nijenhuis tensor. Thereafter, we get
\begin{align*}
    \frac{1}{2}\Delta|\phi|^2\leq -\frac{s}{4}|\phi|^2-C|\phi|^4+C|\phi|^2+C|\beta|^2|\phi|^2+C(\|(F_A)^{1,1}_p\|_{C^0})|\phi|^2
\end{align*}
Now the maximum principle deduces the result. 
\end{proof}
Next we prove some more a priori estimates assuming that $(M,g,J)$ is almost K\"ahler. Lemma $2.1$ in \cite{Partha} tells us that the Clifford action of $(3,0)$-forms on positive spinors is trivial. Moreover $*\beta^{0,3}=i\beta^{0,3}$. Remark \ref{remark Clifford action} says that cl$(\beta^{0,3})\phi^{0,2}=0$. Hence we get cl$(*\beta)\phi=i\mathrm{cl}(\beta^{0,3})\phi^0\in\Omega^{0,3}(\mathcal{L})$, which we can compute using lemma \ref{lemma 2}. In light of the preceding discussion, on an almost K\"ahler $3$-fold $(M,g,J)$, equations \eqref{Dirac 6d} and \eqref{Curvature 6d} take the form:
\begin{align}
    \bar\partial_B \phi^0+\bar\partial_B^* \phi^{0,2}=0\label{Dirac 6 01}\\
    \bar\partial_B\phi^{0,2}+2i\phi^0\beta^{0,3}=0 \label{Dirac 6 03}\\
    \Lambda (F_B)=\frac{i}{8}(3|\phi^0|^2-|\phi^{0,2}|^2-3r)\label{Curvqture 6 1}\\
    F^{0,2}_B+i\bar\partial^*\beta^{0,3}=\frac{\overline{\phi^0}\phi^{0,2}}{4}\label{Curvqture 6 2}
\end{align}
Equation \eqref{Dirac 6 01} represents the $(0,1)$-component of the Dirac equation and \eqref{Dirac 6 03} represents the $(0,3)-$component of the Dirac equation. Equations \eqref{Curvqture 6 1}, \eqref{Curvqture 6 2} represent the curvature equation. A priori, $\Lambda(id^*\beta)$ should appear in equation \eqref{Curvqture 6 1}. However since $\beta\in$Re$(\Omega^{0,3}),\Lambda(d^*\beta)=0$. To see this take an arbitrary function $f\in C^\infty(M).$ Notice 
\begin{align*}
    \int f\Lambda(d^*\beta)dv=\int \langle d^*\beta,f\omega\rangle dv=\int \langle\beta,df\wedge\omega\rangle dv=0
\end{align*}
The last equality happens because $\beta\in\Omega^{0,3}\oplus\Omega^{3,0}$ and $df\wedge\omega\in \Omega^{1,2}\oplus\Omega^{2,1}$.\par 
Notice the last lemma has the term $(F_A)^{1,1}_p$ in the estimate, however this term does not appear in the equations \eqref{Dirac}, \eqref{Curvature}. We would like to prove an estimate without this term; to do this, first we prove a small lemma.
\begin{lemma} \label{norm} On an almost K\"ahler $3$-fold, write
\begin{align*}
    F_B=\frac{1}{3}\Lambda (F_B)\omega+F_B^{2,0}+F_B^{0,2}+(F_B)^{1,1}_p
\end{align*}
Then we have
\begin{align}
    \langle c_1(\mathcal L)^2\cup [\omega],[M]\rangle=\frac{1}{4\pi^2} \Big(\frac{2}{3}\|\Lambda (F_B)\|^2_{L^2}+\|F_B^{2,0}\|^2_{L^2}+\|F_B^{0,2}\|^2_{L^2}-\|(F_B)^{1,1}_p\|^2_{L^2}\Big)
\end{align}
\end{lemma}
\begin{proof}
 On $\Omega^2,\xi\mapsto*(\xi\wedge\omega)$ is an isomorphism between $2$-forms that respects the splitting of $(p,q)$-forms. $(2,0)$ and $(0,2)$ forms are eigen-spaces of eigenvalue $1,$ span$(\omega)$ is an eigen-space of eigenvalue $2,$ $\Omega_p^{1,1}$ is an eigen-space of eigenvalue $-1$. Thereafter,
\begin{align*}
    \langle c_1(\mathcal L)^2\cup [\omega],[M]\rangle&=\Big(\frac{i}{2\pi}\Big)^2\int F_B\wedge F_B\wedge \omega\\
    &=\Big(\frac{i}{2\pi}\Big)^2\int F_B\wedge *^2(F_B\wedge \omega)\\
    &=\Big(\frac{i}{2\pi}\Big)^2 \int F_B\wedge *\Big(\frac{2}{3}\Lambda(F_B)\omega+F_B^{2,0}+F_B^{0,2}-(F_B)^{1,1}_p\Big)\\
    &=-\Big(\frac{i}{2\pi}\Big)^2 \int F_B\wedge *\Big(\overline{\frac{2}{3}\Lambda(F_B)\omega+F_B^{2,0}+F_B^{0,2}-(F_B)^{1,1}_p}\Big)\\
    &=\frac{1}{4\pi^2}\Big(\frac{2}{3}\|\Lambda (F_B)\|^2_{L^2}+\|F_B^{2,0}\|^2_{L^2}+\|F_B^{0,2}\|^2_{L^2}-\|(F_B)^{1,1}_p\|^2_{L^2}\Big) dv\qedhere
\end{align*}
\end{proof}
\begin{proposition}
There exists a constant $C>0$ such that if $(\phi=\phi^0+\phi^{0,2},A,\beta)$ solves the equations \eqref{Dirac 6d}, \eqref{Curvature 6d} on an almost K\"ahler $3$-fold $(M,g,J)$, then we have
\begin{align}
    \int |\phi^0|^4 dv\leq C\int(|\phi^{0,2}|^4+|\beta|^4+r^2) dv
\end{align}
\end{proposition}
\begin{proof}
Equation \eqref{eq 1k} says
\begin{align}
    \frac{1}{2}\Delta|\phi|^2=&\Big(-\frac{s}{4}+2|\beta|^2\Big)|\phi|^2-\frac{1}{2}\langle\Pi(q(\phi))\phi,\phi\rangle\nonumber\\
    &-\frac{1}{2}\langle \mathrm{cl}(\Pi(F_{A^{\mathrm{Ch}}})-rq(\psi_0)+(F_A+2id^*\beta)^{1,1}_p)\phi,\phi\rangle-|\nabla_{A,\beta}\phi|^2\label{start}
\end{align}
We compute and estimate some of the terms appearing above individually.
\begin{comment}
First recall Blair's calculation \cite{Blair} of $\Lambda(iF_{A^{\mathrm{Ch}}})=\frac{s}{2}+\frac{1}{8}|\nabla J|^2$. Integrating this, one gets Blair's beautiful formula
\begin{align*}
    \int \Big(s+\frac{1}{4}|\nabla J|^2\Big) dv=\frac{4\pi}{(m-1)!}\langle c_1(M,J),[\omega]^{m-1}\rangle
\end{align*} 
on an almost K\"ahler manifold of dimension $2n=m$. Therefore the only term in $\mathrm{cl}(F_{A^{\mathrm{Ch}}})$ which contributes to the inner product is $-\frac{i}{2\times 3}\big(s+\frac{1}{4}|\nabla J|^2\big)\omega$. We get
\begin{align}
    &-\frac{s}{4}|\phi|^2-\frac{1}{2}\langle\mathrm{cl}(F_{A^{\mathrm{Ch}}})\phi,\phi\rangle\nonumber
    \\&=-\frac{s}{4}|\phi|^2+\frac{1}{12}\big(s+\frac{1}{4}|\nabla J|^2\big)\langle\mathrm{cl}(i\omega)\phi,\phi\rangle\nonumber\\
    &=-\frac{s}{4}|\phi|^2+\frac{1}{12}\big(s+\frac{1}{4}|\nabla J|^2\big)\langle 3\phi^0-\phi^{0,2},\phi^0+\phi^{0,2}\rangle\qquad[\text{lemma 2.3 from \cite{Partha} is used}]\nonumber\\
    &=\frac{1}{16}|\nabla J|^2|\phi^0|^2-\frac{1}{3}\big(s+\frac{1}{16}|\nabla J|^2\big)|\phi^{0,2}|^2\label{eq 21}
\end{align}
\end{comment}
\begin{align}
    \frac{r}{2}\langle \mathrm{cl}(q(\psi_0))\phi,\phi\rangle&= \frac{r}{8}\langle 3\phi^0-\phi^{0,2},\phi^0+\phi^{0,2}\rangle\nonumber\\
    &=\frac{r}{8}(3|\phi^0|^2-|\phi^{0,2}|^2)
\end{align}
\begin{align}
    &-\frac{1}{2}\langle\mathrm{cl}(2id^*\beta)^{1,1}_p\phi,\phi\rangle\nonumber\\
    &=-\frac{1}{2}\langle\mathrm{cl}(2id^*\beta)^{1,1}_p\phi^{0,2},\phi^{0,2}\rangle\qquad[\text{identity \eqref{Clifford isom} is used}]\nonumber\\
    &\leq C|\beta|^2|\phi^{0,2}|^2\qquad[\text{since }(d^*\beta)^{1,1}=N\circ \beta]
\end{align}
Next we estimate $-|\nabla_{A,\beta}\phi|^2$. First we recall the definition of $\nabla_{\tilde{A}}:=\nabla_{A,\beta}$ from \cite{Fine}. Choose a local coframe $e_j$ which is stationary at a point $x\in M$ with respect to the Levi-Civita connection. We write $\nabla_j$ for the corresponding $\nabla_{\tilde A}$ derivative in the direction dual to $e_j$. We have 
\begin{align*}
    \nabla_{\tilde A}=\nabla_{A}-\frac{1}{2}\sum_j e_j\otimes(\mathrm{cl}(e_j)\, \mathrm{cl}(\tilde\beta)+\mathrm{cl}(\tilde\beta)^*\, \mathrm{cl}(e_j))
\end{align*}
Notice $\mathrm{cl}(\tilde\beta)=\mathrm{cl}(\tilde\beta^*)$ and hence we get
\begin{align*}
    D_{\tilde A}&=\mathrm{cl}\circ\nabla_{\tilde A}\nonumber\\
    &=\mathrm{cl}\circ\nonumber \nabla_A-\frac{1}{2}\sum_j\big(\mathrm{cl}(e_j)^2\,\mathrm{cl}(\tilde\beta)+\mathrm{cl}(e_j)\,\mathrm{cl}(\tilde\beta)\,\mathrm{cl}(e_j)\big)\nonumber\\
    &=D_A+3\mathrm{cl}(\tilde\beta)\qquad[\text{lemma 4.11 in \cite{Fine} says }\sum_j\mathrm{cl}(e_j)\,\mathrm{cl}(\tilde\beta)\,\mathrm{cl}(e_j)=0]
\end{align*}
Since $D_A+\mathrm{cl}(\tilde\beta)=0$, we have $D_{\tilde A}\phi=2\mathrm{cl}(\tilde\beta)\phi$. We get
\begin{align}
     |D_{\tilde{A}}\phi|^2 = \big|\sum_j \mathrm{cl}(e_j)\nabla_j\phi\big|^2\leq 6\sum_j |\mathrm{cl}(e_j)|^2 |\nabla_j\phi|^2=6|\nabla_{\tilde{A}}\phi|^2
\end{align}
Thereafter we have
\begin{align}
    2|\beta|^2|\phi|^2-|\nabla_{A,\beta}\phi|^2&\leq 2|\beta|^2|\phi|^2-\frac{4}{6}|\mathrm{cl}(\tilde\beta)\phi|^2\nonumber\\
    &=2|\beta|^2|\phi|^2-\frac{4}{6}|(\sqrt 2)^3\beta^{0,3}\phi^0|^2\nonumber\\
    &=2|\beta|^2|\phi|^2-\frac{8}{3}|\beta|^2|\phi^0|^2\nonumber\\
    &\leq 2|\beta|^2|\phi^{0,2}|^2
\end{align}
Integrating \eqref{start} over $M$ and using all the computations above we get
\begin{align}
    0\leq&\int\Big(C(|\phi^0|^2+|\phi^{0,2}|^2)+\frac{r}{8}(3|\phi^0|^2-|\phi^{0,2}|^2)+C|\beta|^2|\phi^{0,2}|^2\nonumber\\
    &\hspace{2 ex} -\frac{1}{24}\big(9|\phi^0|^4+|\phi^{0,2}|^4+18|\phi^0|^2|\phi^{0,2}|^2\big)-\frac{1}{2}\langle\mathrm{cl}(F_A)^{1,1}_p\phi^{0,2},\phi^{0,2}\rangle\Big)dv\label{eq start}
\end{align}
We estimate 
\begin{align}
    &\int \Big(-\frac{1}{2}\langle\mathrm{cl}(F_A)^{1,1}_p\phi^{0,2},\phi^{0,2}\rangle\Big)dv\nonumber\\
    &\leq \frac{1}{2}\int\Big(|\mathrm{cl}(F_A)^{1,1}_p|\,|\phi^{0,2}\otimes(\phi^{0,2})^*|\Big)dv\nonumber\\
    &=\int |(F_A)^{1,1}_p|\,|\phi^{0,2}|^2 dv\qquad[\text{since for a two form, } |\mathrm{cl}(\theta)|^2=4|\theta|^2]\nonumber\\
    &\leq 2\int |(F_B)^{1,1}_p|\,|\phi^{0,2}|^2 dv+C\int|\phi^{0,2}|^2 dv\qquad[\text{since }F_A=2F_B+F_{A^{\mathrm{Ch}}}]\nonumber\\
    &=\frac{1}{\epsilon}\int |(F_B)^{1,1}_p|^2 dv+{\epsilon}\int |\phi^{0,2}|^4 dv+C\int|\phi^{0,2}|^2 dv\qquad[\text{Peter-Paul}]\label{eq 25}
\end{align}
Next we estimate $\int |F_B^{0,2}|^2 dv$. Notice that equation \eqref{Curvqture 6 2} gives us 
\begin{align}
    \int (|F_B^{0,2}|^2+|\bar\partial^*\beta^{0,3}|^2+2\mathrm{Re}\langle F_B^{0,2},i\bar\partial^*\beta^{0,3}\rangle) dv=\frac{1}{16}\int |\phi^0|^2|\phi^{0,2}|^2\label{eq 26}
\end{align}
The $(0,3)$ part of $dF_B=0$ is
\begin{align*}
    \bar\partial F_B+\bar{N}\circ F_B^{1,1}=0
\end{align*}
We have $F_B^{1,1}=\frac{1}{3}\Lambda(F_B)\omega+ (F_B)^{1,1}_p$. Since $d\omega=0$, its $(0,3)$-component gives $\bar{N}\circ\omega=0$. Hence, 
\begin{align*}
    \bar{N}\circ F_B^{1,1}=\bar{N}\circ (F_B)^{1,1}_p
\end{align*}
Thereafter we get
\begin{align}
    \int\langle F_B^{0,2},i\bar\partial^*\beta^{0,3}\rangle dv=\int\langle \bar\partial F_B^{0,2},i\beta^{0,3}\rangle dv=-\int \langle \bar{N}\circ (F_B)^{1,1}_p,i\beta^{0,3}\rangle dv\label{eq 27}
\end{align}
We have
\begin{align*}
    &\int |(F_B)^{1,1}_p|^2 dv\\
    &=\int \Big(\frac{2}{3}|\Lambda F_B|^2+2|F_B^{0,2}|^2\Big)dv-4\pi^2\langle c_1(\mathcal L)^2\cup [\omega],[M]\rangle\qquad[\text{lemma \ref{norm} is used}]\\
    &\leq\int\Big(\frac{2}{3\times 64}(3|\phi^0|^2-|\phi^{0,2}|^2-3r)^2+\frac{1}{8}|\phi^0|^2|\phi^{0,2}|^2+4|\langle \bar{N}\circ (F_B)^{1,1}_p,i\beta^{0,3}\rangle|+C\Big)dv\\
    &\hspace{5.8 cm}\qquad[\text{equations \eqref{Curvqture 6 1}, \eqref{Curvqture 6 2}, \eqref{eq 26}, \eqref{eq 27} are used}]\\
    &\leq \int\Big(\frac{1}{96}\big((3|\phi^0|^2+|\phi^{0,2}|^2)^2-6r(3|\phi^0|^2-|\phi^{0,2}|^2)+9r^2\big)+\frac{1}{2}|(F_B)^{1,1}_p|^2+C|\beta|^2+C\Big)dv\\
    &\hspace{11 cm}\qquad[\text{Peter-Paul}]
\end{align*}
which implies
\begin{align}
    \int |(F_B)^{1,1}_p|^2 dv\leq \int \Big(\frac{1}{48}\big((3|\phi^0|^2+|\phi^{0,2}|^2)^2-6r(3|\phi^0|^2-|\phi^{0,2}|^2)+9r^2\big)+C|\beta|^2+C\Big)dv\label{eq 28}
\end{align}
Equation \eqref{eq start} now reads
\begin{align}
    0&\leq\int\Big(C(|\phi^0|^2+|\phi^{0,2}|^2)+\frac{r}{8}(3|\phi^0|^2-|\phi^{0,2}|^2)+C|\beta|^2|\phi^{0,2}|^2\nonumber\\
    &\hspace{1 cm} -\frac{1}{24}\big(9|\phi^0|^4+|\phi^{0,2}|^4+18|\phi^0|^2|\phi^{0,2}|^2\big)-\frac{1}{2}\langle\mathrm{cl}(F_A)^{1,1}_p\phi^{0,2},\phi^{0,2}\rangle\Big)dv\nonumber\\
    &\leq\int\bigg(C|\phi^0|^2+C\big(1+|\beta|^2\big)|\phi^{0,2}|^2-\frac{1}{24}\big(9|\phi^0|^4+|\phi^{0,2}|^4+18|\phi^0|^2|\phi^{0,2}|^2\big)\nonumber\\
    &\hspace{1 cm}\frac{1}{\epsilon}\Big(\frac{1}{48}\big(9|\phi^0|^4+|\phi^{0,2}|^4+6|\phi^0|^2|\phi^{0,2}|^2-6r(3|\phi^0|^2-|\phi^{0,2}|^2)+9r^2\big)+C|\beta|^2+C\Big)\nonumber\\
    &\hspace{1 cm}+\epsilon|\phi^{0,2}|^4+\frac{r}{8}(3|\phi^0|^2-|\phi^{0,2}|^2)\bigg)dv\label{eq final}
\end{align}
Put $\epsilon=1$ in the last inequality \eqref{eq final} and use Peter--Paul again to estimate the terms 
\begin{align*}
    |\phi^0|^2\leq \frac{1}{2\tilde\epsilon}+\frac{\tilde\epsilon}{2}|\phi^0|^4\\
    r|\phi^0|^2\leq \frac{r^2}{2\tilde\epsilon}+\frac{\tilde\epsilon}{2}|\phi^0|^4\\
    |\beta|^2|\phi^{0,2}|^2\leq \frac{1}{2\tilde\epsilon}|\beta|^4+\frac{\tilde\epsilon}{2}|\phi^{0,2}|^4\\
    |\phi^{0,2}|^2\leq \frac{1}{2\tilde\epsilon}+\frac{\tilde\epsilon}{2}|\phi^{0,2}|^4
\end{align*}
Now choose $\tilde\epsilon$ small enough so that a negative multiple of $|\phi^0|^4$ survives on the right hand side. Taking that negative multiple of $|\phi^0|^4$ on the left hand side we get the inequality
\begin{align*}
    \int|\phi^0|^4 dv\leq C\int (|\phi^{0,2}|^4+|\beta|^4+r^2)dv\tag*{\qedhere}
\end{align*}
\end{proof}
\begin{proof}[Proof of theorem \ref{theorem2 symplectic}]
We start with an identity on $\Omega^0(\mathcal L)$ (see \S\S 1.4.2. in \cite{Nicol}):
\begin{align*}
    2\bar\partial_B^*\bar\partial_B\phi^0 =\nabla_B^*\nabla_B\phi^0-i\Lambda(F_B)\phi^0
\end{align*}
Taking the inner product with $\phi^0$ and integrating by parts we get
\begin{align*}
    \int |\nabla_B\phi^0|^2 dv=\int \big(2\langle \bar\partial_B^*\bar\partial_B\phi^0,\phi^0\rangle+ i\Lambda (F_B)|\phi^0|^2\big)dv
\end{align*}
Now use equation \eqref{Dirac 6 01} to deduce
\begin{align*}
    \int 2\langle \bar\partial_B^*\bar\partial_B\phi^0,\phi^0\rangle dv=-\int 2\langle \bar\partial_B^*\bar\partial_B^*\phi^{0,2},\phi^0\rangle dv=-2\int \langle\phi^{0,2},\bar\partial_B^2\phi^{0}\rangle dv
\end{align*}
We have: $
    \bar\partial_B^2\phi^{0}=F_B^{0,2}\phi^0-\bar{N}\circ \partial_B\phi^0$. Therefore we deduce
\begin{align}
    \int |\nabla_B\phi^0|^2 dv&=\int \big(-2\langle \phi^{0,2},F_B^{0,2}\phi^0\rangle+2\langle\phi^{0,2}, \bar{N}\partial_B\phi^0\rangle + i\Lambda (F_B)|\phi^0|^2\big)dv\nonumber\\
    &=\int \big(-2\langle F_B^{0,2}\phi^0,\phi^{0,2}\rangle+ i\Lambda (F_B)|\phi^0|^2+2\overline{\langle\phi^{0,2}, \bar{N}\partial_B\phi^0\rangle}\big)dv\label{32}
\end{align}
Applying $\bar\partial_B^*$ on both sides of equation \eqref{Dirac 6 03} and taking inner product with $\phi^{0,2}$ reads
\begin{align*}
    \langle \bar\partial_B^*\bar\partial_B\phi^{0,2},\phi^{0,2}\rangle+2i\langle\bar\partial_B^*(\phi^0\beta^{0,3}),\phi^{0,2}\rangle=0
\end{align*}
Integration by parts leads us to
\begin{align*}
    \int \big(|\bar\partial_B\phi^{0,2}|^2+2i\langle \beta^{0,3},\overline{\phi^0}\,\bar\partial_B\phi^{0,2}\rangle\big)dv=0 
\end{align*}
Integrating by parts again we deduce that the equation above is equivalent to the following.
\begin{align}
    \int \big(|\bar\partial_B\phi^{0,2}|^2-2i\langle \beta^{0,3},\bar\partial_B(\overline{\phi^0})\,\phi^{0,2}\rangle+2i\langle (\bar\partial^*\beta^{0,3})\phi^0,\phi^{0,2}\rangle\big)dv=0\label{33}
\end{align}
The $(0,3)$-component of the Bianchi identity: $dF_B=0$ gives us $\bar\partial F_B^{0,2}+\bar{N}\circ F_B^{1,1}=0$. Therefore applying $\bar\partial$ on both sides of equation \eqref{Curvqture 6 2} we deduce
\begin{align*}
    -\bar{N}\circ F_B^{1,1}+i\bar\partial\bar\partial^*\beta^{0,3}=\frac{1}{4}\bar\partial\big(\overline{\phi^0}\phi^{0,2}\big)=\frac{1}{4}\big(\bar\partial_B(\overline{\phi^0})\phi^{0,2}+\phi^0\bar\partial_B (\phi^{0,2})\big)
\end{align*}
Taking inner product with $\beta^{0,3}$ on both sides of the equation and integrating by parts we get
\begin{align*}
    -\int\langle \bar{N}\circ F_B^{1,1},\beta^{0,3}\rangle dv+i\int|\bar\partial^*\beta^{0,3}|^2dv =\frac{1}{4}\int\big\langle \big(\bar\partial_B(\overline{\phi^0})\phi^{0,2}+\overline{\phi^0}\bar\partial_B (\phi^{0,2})\big),\beta^{0,3}\big\rangle dv
\end{align*}
Next we use equation \eqref{Dirac 6 03} to simplify the above equality into the following:
\begin{align}
    \int|\bar\partial^*\beta^{0,3}|^2=-\frac{i}{4}\int\big\langle \big(\bar\partial_B(\overline{\phi^0})\phi^{0,2},\beta^{0,3}\big\rangle dv-\frac{1}{2}\int|\phi^0|^2|\beta^{0,3}|^2 dv-i\int\langle \bar{N}\circ F_B^{1,1},\beta^{0,3}\rangle dv\label{34}
\end{align}
Combining equations \eqref{32}, \eqref{33} and \eqref{34} we get
\begin{align*}
    &\int \Big(|\nabla_B\phi^0|^2+|\bar\partial_B\phi^{0,2}|^2+8|\bar\partial ^*\beta^{0,3}|^2+4|\phi^0|^2|\beta^{0,3}|^2+\frac{1}{2}|\phi^0|^2|\phi^{0,2}|^2-i\Lambda (F_B)|\phi^0|^2\Big) dv\\
    &=\int \Big(2\langle \bar{N}\circ \partial_B\phi^0,\phi^{0,2}\rangle -8i\langle \bar{N}\circ F_B^{1,1},\beta^{0,3}\rangle +4i\langle \beta^{0,3},\bar\partial_B(\overline{\phi^0})\phi^{0,2}\rangle\Big)dv
\end{align*}
Replacing $i\Lambda(F_B)=-\frac{1}{8}(3|\phi^0|^2-|\phi^{0,2}|^2-3r)$ from equation \eqref{Curvqture 6 1} we get
\begin{align*}
    &\int \Big(|\nabla_B\phi^0|^2+|\bar\partial_B\phi^{0,2}|^2+8|\bar\partial ^*\beta^{0,3}|^2+4|\phi^0|^2|\beta^{0,3}|^2\\
    &\hspace{4 ex}+\frac{1}{2}|\phi^0|^2|\phi^{0,2}|^2+\frac{1}{8}(3|\phi^0|^2-|\phi^{0,2}|^2-3r)|\phi^0|^2\Big) dv\\
    &=\int \Big(2\langle \bar{N}\circ\partial_B\phi^0,\phi^{0,2}\rangle -8i\langle \bar{N}\circ F_B^{1,1},\beta^{0,3}\rangle +4i\langle \beta^{0,3},\bar\partial_B(\overline{\phi^0})\phi^{0,2}\rangle\Big)dv
\end{align*}
which in turn reads
\begin{align}
    &\int \Big(|\nabla_B\phi^0|^2+|\bar\partial_B\phi^{0,2}|^2+8|\bar\partial ^*\beta^{0,3}|^2+4|\phi^0|^2|\beta^{0,3}|^2\nonumber\\
    &\hspace{4 ex}+\frac{3}{8}|\phi^0|^2|\phi^{0,2}|^2+\frac{3}{8}(|\phi^0|^2-r)^2+\frac{3r}{8}(|\phi^0|^2-r)\Big) dv\nonumber\\
    &=\int \Big(2\langle \bar{N}\circ\partial_B\phi^0,\phi^{0,2}\rangle -8i\langle \bar{N}\circ F_B^{1,1},\beta^{0,3}\rangle +4i\langle \beta^{0,3},\bar\partial_B(\overline{\phi^0})\phi^{0,2}\rangle\Big)dv\label{eq 1000}
\end{align}
Since $\mathcal L$ is chosen to be the trivial line bundle we also have \begin{align*}
    0=\int \Lambda(F_B) dv=\frac{i}{8}\int (3|\phi^0|^2-|\phi^{0,2}|^2-3r)dv
\end{align*}
which implies
\begin{align*}
\frac{3r}{8}\int (|\phi^0|^2-r) dv=\frac{r}{8}\int |\phi^{0,2}|^2 dv
\end{align*}
Replacing this in \eqref{eq 1000} we get
\begin{align*}
        &\int \Big(|\nabla_B\phi^0|^2+|\bar\partial_B\phi^{0,2}|^2+8|\bar\partial ^*\beta^{0,3}|^2+4|\phi^0|^2|\beta^{0,3}|^2\nonumber\\
    &\hspace{4 ex}+\frac{3}{8}|\phi^0|^2|\phi^{0,2}|^2+\frac{3}{8}(|\phi^0|^2-r)^2+\frac{r}{8} |\phi^{0,2}|^2\Big) dv\nonumber\\
    &=\int \Big(2\langle \bar{N}\circ\partial_B\phi^0,\phi^{0,2}\rangle -8i\langle \bar{N}\circ F_B^{1,1},\beta^{0,3}\rangle +4i\langle \beta^{0,3},\bar\partial_B(\overline{\phi^0})\phi^{0,2}\rangle\Big)dv
\end{align*}
Under the assumption of $\beta=0$, we have
\begin{align*}
    &\int \Big(|\nabla_B\phi^0|^2+|\bar\partial_B\phi^{0,2}|^2+\frac{3}{8}|\phi^0|^2|\phi^{0,2}|^2+\frac{3}{8}(|\phi^0|^2-r)^2+\frac{r}{8}|\phi^{0,2}|^2\Big) dv\\&=2\int \langle \bar{N}\circ\partial_B\phi^0,\phi^{0,2}\rangle dv
\end{align*}
We estimate the right hand side using Peter-Paul inequality to obtain
\begin{align*}
    &\int \Big(|\nabla_B\phi^0|^2+|\bar\partial_B\phi^{0,2}|^2+\frac{3}{8}|\phi^0|^2|\phi^{0,2}|^2+\frac{3}{8}(|\phi^0|^2-r)^2+\frac{r}{8}|\phi^{0,2}|^2\Big) dv\\&\leq \int \Big( \frac{1}{2}|\nabla_B\phi^0|^2+C|\phi^{0,2}|^2\Big) dv
\end{align*}
which in turn gives us
\begin{align*}
    &\int \Big(\frac{1}{2}|\nabla_B\phi^0|^2+|\bar\partial_B\phi^{0,2}|^2+\frac{3}{8}|\phi^0|^2|\phi^{0,2}|^2+\frac{3}{8}(|\phi^0|^2-r)^2+\big(\frac{r}{8}-C\big)|\phi^{0,2}|^2\Big) dv\leq 0
\end{align*}
Therefore for $r>8C$, we have
\begin{align*}
    \nabla_B\phi^0=0,\, \phi^{0,2}=0,\,|\phi^0|=\sqrt{r}
\end{align*}
Hence for $(M,g,J)$ an almost K\"ahler $3$-fold, the canonical spin$^\mathbb C$-structure gives us a unique solution modulo gauge under the assumption $\beta=0$: $\phi=\sqrt{r}, B=$ the trivial connection on the trivial line bundle and $\beta=0$.
\end{proof}
\begin{comment}
\begin{proof}[Proof of theorem \ref{transversal}]
For the canonical spin bundles, the equations \eqref{Dirac}, \eqref{Curvature} are elliptic modulo gauge with index zero. Therefore to prove that the canonical solution is nondegenerate, all we have to show is that the kernel of the linearised equations at the canonical solution trivial.\par Suppose $\delta(A^{\mathrm{Ch}},0,\sqrt{r}\psi_0)=(2ia,b,\sigma)$ be an infinitesimal perturbation of the canonical solution where $a\in\Omega^1,b=b^3+b^5+\cdots+b^{\big(m-\frac{1+(-1)^m}{2}\big)}\in\oplus_{k>1}$Re$(\Omega^{0,2k+1}),\sigma\in \Gamma(S^{\mathrm{can}}_+)$. $(2ia,b,\sigma)$ are in the kernel of the linearised equations if they satisfy the following equations.
\begin{align}
    D_{A^{\mathrm{Ch}}}\sigma+\sqrt{r}\,\mathrm{cl}(ia+b)\psi_0=0\label{Lin Dirac}\\
    \Pi\big(2ida+F_b\big)= \sqrt{r}\,\Pi\Big(\psi_0^*\otimes \sigma+\sigma^*\otimes \psi_0-\frac{\langle\psi_0,\sigma\rangle+\langle\sigma,\psi_0\rangle}{2^{m-1}}\Big)\label{Lin Curvature}
\end{align}
Write 
\begin{align*}
    a&=\frac{1}{\sqrt{2}}(b_1+\overline{b_1}),\, b^k=\frac{1}{(\sqrt{2})^k}(b_k+\overline{b_k})
\end{align*}
where $b_k\in\Omega^{0,k}$.
\end{proof}
\end{comment}
\subsection{An energy identity}
In this section we prove theorem \ref{Energy}. Before we go into the proof, we prove a small lemma.
\begin{lemma}\label{lemma Cl}
On a $6$-dimensional spin$^\mathbb C$-manifold, let $\theta\in i\Omega^2$, then for any $\phi\in\Gamma(S_+)$,
\begin{align*}
    \langle\mathrm{cl}(\theta)\phi,\phi\rangle=4\langle\theta, q(\phi)\rangle
\end{align*}
\end{lemma}
\begin{proof}
    In dimension $6$, Clifford multiplication gives us an isomorphism 
    \begin{align*}
        \mathrm{cl}:i\Lambda^2\rightarrow i\mathfrak{su}(S_+)
    \end{align*}
    which corresponds to isomorphism between two irreducible Spin$(6)$ representations.  It follows from Schur's lemma that there is a  constant $a$ such that for all $\theta\in i\Omega^2,\phi\in\Gamma(S_+),$
    \begin{align*}
        |\mathrm{cl}(\theta)|^2=a|\theta|^2
    \end{align*}
    As shown in the proof of theorem \ref{vortex}, $a=4$. We now compute
\begin{align*}
    4\langle \theta, q(\phi)\rangle&=\big\langle \mathrm{cl}(\theta), E_\phi\big\rangle\\
    &=\mathrm{Tr}\big( \mathrm{cl}(\theta)\circ\phi^{*}\otimes \phi\big)\\
    &=\sum_i\big\langle \mathrm{cl}(\theta) \langle e_i,\phi\rangle\phi, e_i\big\rangle\\
    &=\big\langle \mathrm{cl}(\theta)\phi,\phi\big\rangle 
\end{align*}
In the first line, we use the fact that $a=4$, so that $\mathrm{cl}$ scales the inner product by $4$.  In the second line we use the fact that $\mathrm{cl}(\theta)$ is trace-free and so we can replace $E_\phi$ by $\phi^{*}\otimes \phi$; recall \eqref{Ephi} for $E_\phi$.
\end{proof}
\begin{proof}[Proof of theorem \ref{Energy}]
The Weitzenb\"ock formula \eqref{Weitzenbock}, together with \eqref{Weitzenbock 6} give us
\begin{align*}
    &\int |D_{A,\beta}\phi|^2 dv\\
    &=\int \Big(|\nabla_{A,\beta}\phi|^2+\frac{s}{4}|\phi|^2+\frac{1}{2}\big\langle\mathrm{cl}(F_A+2id^*\beta)\phi,\phi\big\rangle-2|\beta|^2|\phi|^2\Big) dv\\
    &=\int \Big(|\nabla_{A,\beta}\phi|^2+\frac{s}{4}|\phi|^2+\big\langle\mathrm{cl}(F_B+id^*\beta)\phi,\phi\big\rangle+\frac{1}{2}\big\langle \mathrm{cl}(F_{A^{\mathrm{Ch}}})\phi,\phi\big\rangle-2|\beta|^2|\phi|^2\Big) dv
\end{align*}
By lemma \ref{lemma Cl},
\begin{align*}
    \big\langle\mathrm{cl}(F_B+id^*\beta)\phi,\phi\big\rangle=4\big\langle F_B+id^*\beta, q(\phi)\big\rangle
\end{align*}
Thereafter we have
\begin{align*}
    &\int \Big(|D_{A,\beta}\phi|^2+4|\Pi\big(F_B+id^*\beta+\frac{ir}{8}\omega-\frac{1}{2}q(\phi)\big)|^2\Big) dv\\
    &=\int \Big(|D_{A,\beta}\phi|^2+ 4|\Pi (F_B+id^*\beta+\frac{ir}{8}\omega)|^2+|\Pi(q(\phi))|^2-4\big\langle\Pi(F_B+id^*\beta+\frac{ir}{8}\omega), q(\phi)\big\rangle\Big) dv\\
    &=\int \Big(|\nabla_{A,\beta}\phi|^2+\frac{s}{4}|\phi|^2+4|\Pi(F_B+id^*\beta+\frac{ir}{8}\omega)|^2+\big\langle 4(F_B)^{1,1}_p+2F_{A^\mathrm{Ch}},q(\phi)\big\rangle\\
    &\hspace{5 ex} -\frac{1}{8} \big\langle \mathrm{cl}(ir\omega)\phi,\phi\big\rangle +|\Pi(q(\phi))|^2-2|\beta|^2|\phi|^2\Big) dv\qquad[\text{lemma \ref{lemma Cl} is used}]
\end{align*}
Lemma \ref{norm} gives us
\begin{align*}
    \int |(F_B)_p^{1,1}|^2 dv=-4\pi^2\langle c_1(\mathcal L)^2\cup [\omega],[M]\rangle+\int \frac{2}{3} \Big(|\Lambda(F_B)|^2+|F_B^{2,0}|^2+|F_B^{0,2}|^2\Big) dv
\end{align*}
Therefore,
\begin{align*}
    \int|F_B|^2 dv&=\int \Big(\frac{1}{3}|\Lambda (F_B)|^2+|F_B^{2,0}|^2+|F_B^{0,2}|^2+|(F_B)_p^{1,1}|^2\Big) dv\\
&=-4\pi^2\langle c_1(\mathcal L)^2\cup [\omega],[M]\rangle+\int \Big(|\Lambda(F_B)|^2+2|F_B^{2,0}|^2+2|F_B^{0,2}|^2\Big) dv\\
&=-4\pi^2\langle c_1(\mathcal L)^2\cup [\omega],[M]\rangle+\frac{1}{3}\int |\Lambda(F_B)|^2 dv+2\int |\Pi(F_B)|^2 dv
\end{align*}
Also since $\beta\in$ Re$\Omega^{0,3},\Pi(id^*\beta)=id^*\beta$. Therefore,
\begin{align*}
    \langle \Pi(F_B),\Pi(id^*\beta)\rangle_{L^2}&=\langle F_B,\Pi(id^*\beta)\rangle_{L^2}\\
    &=\langle F_B,id^*\beta\rangle_{L^2}\\
    &=\langle dF_B,i\beta\rangle_{L^2}\\
    &=0
\end{align*}
Similarly, 
\begin{align*}
    \langle id^*\beta, i\omega\rangle_{L^2}=0\hspace{1 ex}\mathrm{since  }\hspace{1 ex}d\omega=0
\end{align*}
and 
\begin{align*}
    8\langle \Pi(F_B),\frac{ir}{8}\omega\rangle_{L^2}=-2\pi r\int \frac{i}{2\pi}F_B\wedge \frac{\omega^2}{2!}=-2\pi r\mathrm{deg}(\mathcal L)
\end{align*}
We also compute the terms
\begin{align*}
    \int |\Pi(q(\phi))|^2 dv &=\int \langle \Pi(q(\phi)),q(\phi)\rangle dv\\
    &=\frac{1}{4}\int \langle \mathrm{cl}\big(\Pi(q(\phi))\big)\phi,\phi\rangle dv\\
    &=\frac{1}{48}\int(9|\phi^0|^4+|\phi^{0,2}|^4+18|\phi^0|^2|\phi^{0,2}|^2) dv\\
    &=\int\Big(\frac{1}{48}(3|\phi^0|^2-|\phi^{0,2}|^2)^2+\frac{1}{2}|\phi^0|^2|\phi^{0,2}|^2\Big)dv
\end{align*}
and 
\begin{align*}
    -\frac{1}{8}\langle \mathrm{cl}(ir\omega)\phi,\phi\rangle&=-\frac{r}{8}\langle 3\phi^0-\phi^{0,2},\phi^0+\phi^{0,2}\rangle\\
    &=-\frac{r}{8}(3|\phi^0|^2-|\phi^{0,2}|^2)
\end{align*}
Finally assembling all the pieces together we get
\begin{align*}
    &\int \Big(|D_{A,\beta}\phi|^2+4|\Pi\big(F_B+id^*\beta+\frac{ir}{8}\omega-\frac{1}{2}q(\phi)\big)|^2\Big) dv\\
    &=\int\Big(|\nabla_{A,\beta}\phi|^2+\frac{s}{4}|\phi|^2+2|F_B|^2+4|d^{*}\beta|^2+\frac{3r^2}{16}+\frac{1}{48}(3|\phi^0|^2-|\phi^{0,2}|^2)^2+\frac{1}{2}|\phi^0|^2|\phi^{0,2}|^2\\
    &\hspace{5 ex}-\frac{2}{3}|\Lambda(F_B)|^2-2|\beta|^2|\phi|^2-\frac{r}{8}(3|\phi^0|^2-|\phi^{0,2}|^2)+\big\langle 4(F_B)^{1,1}_p+2F_{A^\mathrm{Ch}},q(\phi)\big\rangle\Big)dv\\
    &\hspace{5 ex}-2\pi r\mathrm{deg}(\mathcal L)+8\pi^2\langle c_1(\mathcal L)^2\cup [\omega],[M]\rangle\\
    &=\mathcal E(A,\beta,\phi)+8\pi^2\langle c_1(\mathcal L)^2\cup [\omega],[M]\rangle
\end{align*}
The left hand side is non-negative and equal to zero if and only if $(A,
\beta,\phi)$ solve the equations \eqref{Dirac}, \eqref{Curvature} in dimension $n=6$.
\end{proof}
\printbibliography[
heading=bibintoc,
title={Bibliography}
]

%\begin{thebibliography}{9}% Replace 9 by 99 if 10 or more references
%
% Please note the use of "\and" between author names below
%
%bibitem{Goddard}
%{\bibname P. Goddard, A. Kent \and D. I. Olive}, `Unitary
%representations of the Virasoro and Supervirasoro algebras', {\em
%Comm. Math. Phys. }103 (1986) 105.
%

%\end{thebibliography}

%\affiliationone{% in this example, two authors share an institution
   %Partha Ghosh\\
    %IMJ-PRG}
   %\email{Partha.Ghosh@imj-prg.fr}
% Important: Do not put any empty line here.
%\affiliationtwo{% in this example, one author has two addresses}
   %T. Hird\\
   %Previous postal address where
    % the research was performed and\\
   %Country
   %\email{hird@university.ac.uk}}
% Important: Do not put any empty line here.
% Use \affiliationthree{} for any address positioned under \affiliationone
% Use \affiliationfour{}  for any address positioned under \affiliationtwo
%\affiliationthree{~} %inserts a space to make this field empty
%\affiliationfour{%
   %Current address:\\
   %Present long-term address\\
   %Country
   %\email{t.hird@institution.edu}}
%

%\begin{keywords}
%Sections; lists; figures; tables; mathematics; fonts; references; appendices
%\end{keywords}

\Addresses
\end{document}